\documentclass[11pt, reqno]{amsart}
\usepackage[dvipsnames]{xcolor}
\usepackage{geometry, amsmath, amsfonts, amssymb, tikz-cd, mathtools,lscape,multirow}
\usepackage{mathrsfs}
\usepackage{graphicx, rotating}
\usepackage{cancel}
\usepackage[bordercolor=gray!20,backgroundcolor=blue!10,linecolor=blue!50,textsize=footnotesize,textwidth=0.8in]{todonotes}
\usepackage{hyperref}
\usepackage{caption}
\usepackage{subcaption}

\newtheorem{thm}{Theorem}[section]
\newtheorem{theorem}[thm]{Theorem}

\newtheorem{corollary}[thm]{Corollary}

\newtheorem{lemma}[thm]{Lemma}

\newtheorem{proposition}[thm]{Proposition}

\theoremstyle{definition}

\newtheorem{definition}[thm]{Definition}

\newtheorem{example}[thm]{Example}

\newtheorem{question}[thm]{Question}
\newtheorem{remark}[thm]{Remark}
\newtheorem{convention}[thm]{Convention}

\numberwithin{equation}{section}

\newcommand{\R}{\mathbb{R}}
\newcommand{\N}{\mathbb{N}}
\newcommand{\Q}{\mathbb{Q}}
\newcommand{\Z}{\mathbb{Z}}
\newcommand{\tr}{{\rm{tr}}}
\newcommand{\ol}{\overline}

\newcommand{\MCG}{{\rm{MCG}}}
\newcommand{\s}{\sigma}

\newcommand{\T}{\mathbb{T}}
\newcommand{\SL}{\rm{SL}}
\newcommand{\rs}{\stackrel{\rm r}{\rightharpoonup}} 
\newcommand{\ls}{\stackrel{\rm l}{\rightharpoonup}} 

\newcommand{\rh}{\rightharpoonup}

\newcommand{\tz}{\tau_{[\frac{0}{1}, \frac{1}{0}]}} 
\newcommand{\vs}{\left(\begin{smallmatrix} 1 \\ s \end{smallmatrix}\right)} 
\newcommand{\vxy}{\left(\begin{smallmatrix} x \\ y \end{smallmatrix}\right)} 
\newcommand{\abcd}{\left(\begin{smallmatrix}a & c \\ b & d \end{smallmatrix}\right)} 
\newcommand{\I}{\rm{I}}
\newcommand{\f}{{\bar{f}}}

\begin{document}

\title{Complete description of Agol cycles of pseudo-Anosov $3$-braids}

\author{Keiko Kawamuro}
\address {Department of Mathematics, University of Iowa, Iowa City, IA 52242}
\email{keiko-kawamuro@uiowa.edu}

\author{Eiko Kin} 
\address {Center for Education in Liberal Arts and Sciences, Osaka University, Toyonaka, Osaka 560-0043, Japan}
\email{kin.eiko.celas@osaka-u.ac.jp}


\date{\today}

\maketitle

\begin{abstract}
The equivalence class of an Agol cycle is a conjugacy invariant of a pseudo-Anosov map. 
Mosher defined train tracks in the torus associated to Farey intervals and 
investigated 
the relation between the train tracks and the continued fraction expansions of quadratic irrational numbers.   
We study Mosher's train tracks and 
describe Agol cycles of all the pseudo-Anosov $3$-braids.
\end{abstract}

\section{Introduction}

Let $\Sigma= \Sigma_{g, n}$ be an orientable surface with genus $g$ and $n$ punctures. 
Let $\MCG(\Sigma)$ be the mapping class group of $\Sigma$. 
By the Nielsen-Thurston classification, any homeomorphism $\phi : \Sigma \rightarrow \Sigma$ is isotopic to a homeomorphism 
which is either {\em periodic}, {\em reducible} or {\em pseudo-Anosov} \cite{FarbMargalit12, thurston:mappingtori}. 
If $\phi$ is a pseudo-Anosov map there exist  stable and unstable measured laminations 
$(\mathcal{L}^{\tt s}, \nu^{\tt s})$ and $(\mathcal{L}^{\tt u}, \nu^{\tt u})$
and the dilatation $\lambda>1$ such that 
$$\phi(\mathcal{L}^{\tt s}, \nu^{\tt s})= (\mathcal{L}^{\tt s}, \lambda \nu^{\tt s}) \hspace{2mm} \mbox{and} \hspace{2mm} 
\phi(\mathcal{L}^{\tt u}, \nu^{\tt u})= (\mathcal{L}^{\tt u}, \lambda^{-1} \nu^{\tt u}).$$

A {\em train track} $\tau$ is a finite embedded $C^1$ graph in  the surface $\Sigma$ 
equipped with a well-defined tangent line at each vertex. It also requires that no component of $\Sigma \setminus \tau$ is an immersed nullgon, monogon, bigon, once-punctured nullgon, or annulus \cite{PennerHarer92}. 
In this paper following Agol \cite{Agol11} we exclusively study trivalent train tracks, 
i.e.,  train tracks all of whose vertices have valence $3$. 
Edges of a train track are called {\em branches} and vertices  are called {\em switches}.  
Figure~\ref{fig_switch-omega}-(2),(3) illustrates a trivalent train track $\omega_0$ in $\Sigma_{0,4}$ such that 
each component of $\Sigma_{0,4} \setminus \omega_0$ is a once-punctured monogon.

A branch $e$ is {\em large at a switch $v$} of $e$ 
if in a small neighborhood of $v$ each immersed arc in the train track through $v$ intersects the interior of $e$. 
Otherwise $e$ is {\em small}  {\em at the switch $v$}. 
A branch $e$ is {\em large} (resp. {\em small}) 
if $e$ is large (resp. small) at the both switches.  
Otherwise $e$ is {\em mixed}. 

A {\em measured train track} $(\tau, \mu)$ in $\Sigma$ is a train track $\tau$ together with a transverse measure $\mu$. 
The transverse measure $\mu$ is a function which assigns a weight to each branch. At every switch the weights satisfy the {\em switch condition}: the sums of the weights on each side of the switch are equal to each other. See the equality $\mu(c)= \mu(a)+ \mu(b)$ in Figure~\ref{fig_switch-omega}-(1).

\begin{figure}[t]
\centering
\includegraphics[height=4cm]{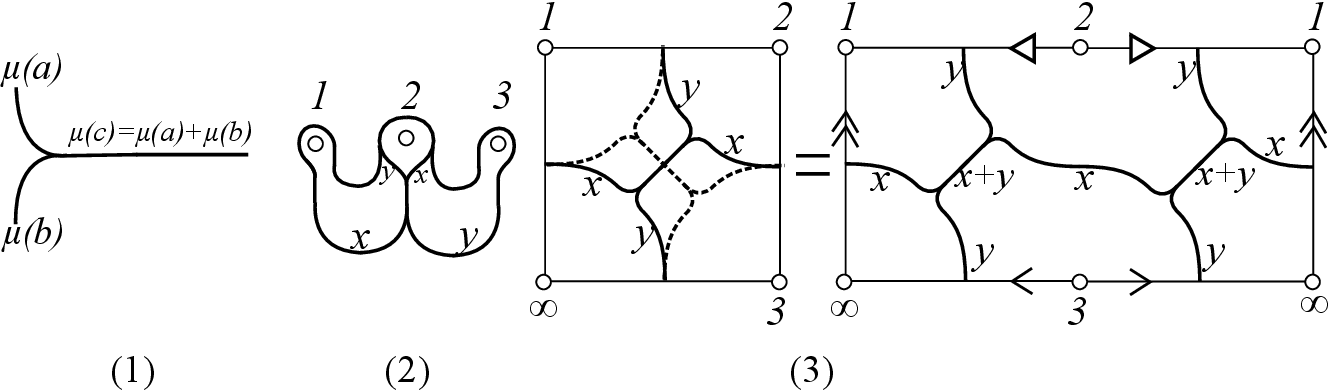}
\caption{(1) Switch condition. 
(2), (3) The measured train track $(\omega_0, \vxy)$ in $\Sigma_{0, 4}$. 
In (3) $\Sigma_{0,4}$ is viewed as a square pillowcase with the corners removed. Dotted arcs are branches on the backside.}
\label{fig_switch-omega}
\end{figure}

Left/right splitting, folding and shifting are operations on $(\tau, \mu)$ that give 
a new measured train track. 
\begin{definition}
\label{definition:3operations}
\begin{enumerate}
\item 
Figure~\ref{fig_split_shift}-(1) shows a neighborhood of a large branch of $(\tau, \mu)$
whose weight is $x+y= z+w$ by the switch condition. 
If $z>x$ equivalently $y>w$ (resp. $x>z$ equivalently $w>y$)
the measured train track $(\tau, \mu)$ admits a {\em left} (resp. {\em right}) {\em splitting} at the large branch.  
\item 
A {\em shifting} (at a mixed branch) is an operation on $(\tau, \mu)$ 
as depicted in  Figure~\ref{fig_split_shift}-(2). 
\end{enumerate}
\begin{figure}[ht]
\centering
\includegraphics[width=8.2cm]{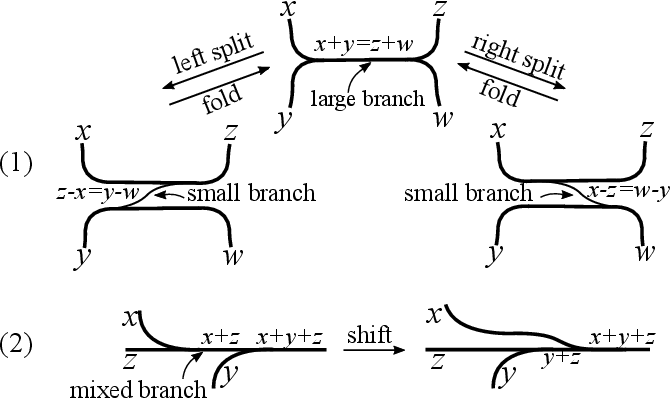}
\caption{(1) Left splitting when $z>x$ ($\Leftrightarrow y>w$), right splitting  when $x>z$ ($\Leftrightarrow w>y$) at a large branch, and folding at a small branch.
(2) Shifting at a mixed branch.}
\label{fig_split_shift}
\end{figure}
\end{definition}

A {\em maximal splitting} \cite{Agol11} of a measured train track $(\tau, \mu)$ is a set of  simultaneous splittings along all the large branches of maximal $\mu$-weight. 
We denote $(\tau, \mu)\rightharpoonup (\tau', \mu')$ 
if $(\tau', \mu')$ is the result of the maximal splitting.
Consecutive $n$ maximal splittings
$$(\tau, \mu) \rightharpoonup (\tau_1, \mu_1) \rightharpoonup \cdots \rightharpoonup (\tau_n, \mu_n) $$ 
is denoted by $(\tau,\mu) \rightharpoonup^n (\tau_n, \mu_n)$. 
If all the splittings in a maximal splitting are of left (resp. right) type then we write  $(\tau, \mu)\ls (\tau', \mu')$ (resp. $(\tau, \mu)\rs (\tau', \mu')$) and 
say that $(\tau, \mu)$ admits a {\em left} (resp. {\em right}) {\em maximal splitting}.

Let $\phi: \Sigma \to \Sigma$ be a homeomorphism. 
For a measured train track $(\tau, \mu)$ in $\Sigma$, we define a new measured train track 
$\phi(\tau, \mu)$ in $\Sigma$ by 
$\phi (\tau, \mu) := ( \phi(\tau), \phi_*(\mu))$, 
where the measure $\phi_*(\mu)$ is defined by $\phi_*(\mu)(e):= \mu(\phi^{-1}(e))$ for every branch $e$ in the train track $\phi(\tau)$. 

The following theorem by Agol is a starting point of our study.

\begin{theorem}[Theorem~3.5 in \cite{Agol11}]
\label{thm:Agol}
Let $\phi: \Sigma \to \Sigma$ be a pseudo-Anosov map with dilatation $\lambda$. 
Let $(\tau, \mu)$ be a measured train track 
suited to the stable measured lamination of $\phi$. 
Then  there exist  $n\geq 0$ and $m >0$ such that 
$$(\tau, \mu)  \rightharpoonup^n (\tau_n, \mu_n) \rightharpoonup^m (\tau_{n+m}, \mu_{n+m}),$$ 
where 
$(\tau_{n+m}, \mu_{n+m})= \phi(\tau_n, \lambda^{-1}\mu_n) = (\phi(\tau_n), \lambda^{-1} \phi_*(\mu_n)).$
\end{theorem}

See also Agol-Tsang \cite{AgolTsang22}. 
We note that in \cite{Wu18}  Wu worked out orbifold theory and gave more efficient encoding of splitting sequences for braids.

For the terminology {\em suited to,} see Definition~\ref{definition:suited}. 
Theorem~\ref{thm:Agol} states that the maximal splitting sequence is eventually periodic 
modulo the action of $\phi$ and a scaling by the dilatation. 
We call the maximal splitting sequence 
$$ (\tau_n, \mu_n) \rightharpoonup^m (\tau_{n+m}, \mu_{n+m}) = \phi(\tau_n, \lambda^{-1}\mu_n) \rightharpoonup  \cdots$$ 
a {\em periodic splitting sequence} of $\phi$. 
We may call the finite subsequence 
$$(\tau_n, \mu_n) \rightharpoonup^m (\tau_{n+m}, \mu_{n+m}) =  \phi(\tau_n, \lambda^{-1}\mu_n)$$ 
an {\em Agol cycle} of $\phi$ and say that  the {\em length} of the Agol cycle is $m$. 
Clearly, the maximal splitting sequence 
$(\tau_{n+1}, \mu_{n+1}) \rightharpoonup^m (\tau_{n+m+1}, \mu_{n+m+1})  \rightharpoonup \cdots$ 
starting at $(\tau_{n+1}, \mu_{n+1})$ is also a periodic splitting sequence of $\phi$, 
and the finite sequence 
$(\tau_{n+1}, \mu_{n+1}) \rightharpoonup^m (\tau_{n+m+1}, \mu_{n+m+1})$ is  an Agol cycle of $\phi$ as well.

\begin{definition}
\label{definition_ combinatorially-isomorphic}
Let $\phi, \phi':\Sigma \to \Sigma$ be pseudo-Anosov maps with 
periodic splitting sequences 
$$ \mathscr{P}: (\tau_n, \mu_n) \rh^m (\tau_{n+m}, \mu_{n+m})= \phi(\tau_n, \lambda^{-1}\mu_n)\  \rightharpoonup  \cdots$$
 of  $\phi$ and 
$$\mathscr{P}': (\tau'_{n'}, \mu'_{n'}) \rh^{m'} (\tau'_{n'+m'}, \mu'_{n'+m'})= \phi'(\tau'_{n'},  (\lambda')^{-1}\mu'_{n'})\ \rightharpoonup  \cdots$$ 
of $\phi'$. 
We say that $\mathscr{P}$ and $\mathscr{P}'$ are {\em combinatorially isomorphic} (\cite{HodgsonIssaSegerman16})  
if $m=m'$ and
there exists  an orientation-preserving diffeomorphism 
$h:\Sigma \to \Sigma$, integers $p, q \in \Z_{\geq 0}$ and $c\in \R_{>0}$ such that the following conditions (1) and (2) hold.
\begin{enumerate}
\item[(1)] 
$\phi'=h \circ \phi \circ h^{-1}$. 

\item[(2)] 
$h(\tau_{p+i}, \mu_{p+i})=(\tau'_{q+i}, c\mu'_{q+i})$ for all $i\in \Z_{\geq 0}$.
\end{enumerate} 
We say that Agol cycles  
$(\tau_n, \mu_n) \rh^m (\tau_{n+m}, \mu_{n+m})$ of $\phi$ and $(\tau'_{n'}, \mu'_{n'}) \rh^{m'} (\tau'_{n'+m'}, \mu'_{n'+m'})$ of $\phi'$ are {\em equivalent} if $m=m'$ and there exists a diffeomorphism 
$h:\Sigma \to \Sigma$,  integers $p \geq n,$ and $p' \geq n'$, and a number $c>0$ such that 
$h(\tau_{n+p}, \mu_{n+p})=(\tau'_{n'+ p'}, c\mu'_{n'+ p'})$. 
This implies the above condition (2), see Lemma~\ref{lemma_comb_eq}.
\end{definition}

By the definition, Agol cycles of $\phi$ and $\phi'$ are equivalent  if 
periodic splitting sequences $\mathscr{P}$ of $\phi$ and $\mathscr{P}'$ of $\phi'$ are combinatorially isomorphic. 
Hodgson-Issa-Segerman proved the following.

\begin{theorem}[Theorem~5.3 in  \cite{HodgsonIssaSegerman16}] 
\label{theorem_Hodgson-Issa-Segerman}

The mapping classes of 
pseudo-Anosov maps 
$\phi, \phi':\Sigma \to \Sigma$ are conjugate in $\MCG(\Sigma)$ 
if and only if $\mathscr{P}$ and $\mathscr{P}'$  are combinatorially isomorphic. 
\end{theorem}
As a consequence,  the equivalence class of an Agol cycle of $\phi$ is a conjugacy invariant of the pesudo-Anosov map $\phi$. 
The length of an Agol cycle of $\phi$ is also a conjugacy invariant and we  call it the {\em Agol cycle length} of $\phi$.

Each element of the $3$-braid group $B_3$ induces an element of the mapping class group $\mathrm{MCG}(\Sigma_{0,4})$ of a $4$-punctured sphere 
(Section~\ref{subsection_braid-mappingclass}). 
By Murasugi's classification of $3$-braids \cite[Proposition~2.1]{Murasugi74},  
a braid $b \in B_3$ is pseudo-Anosov (i.e. a representative of the corresponding mapping class in $\mathrm{MCG}(\Sigma_{0,4})$ is pseudo-Anosov) if and only if $b$ is conjugate to 
a braid 
$\Delta^{2j} \s_1^{p_1} \s_2^{-q_1} \cdots \s_1^{p_k} \s_2^{-q_k}$ 
where $\Delta=\s_1\s_2\s_1$, $j$ is an integer, and $p_1, q_1, \cdots, p_k, q_k $ and $k$ are positive integers. 
Moreover, $j$ and $k$ are unique and the pairs $(p_1, q_1), \cdots, (p_k, q_k)$ are unique up to cyclic permutation.

In this paper we give a complete description of the Agol cycles of pseudo-Anosov $3$-braids:

\begin{theorem}[cf. Theorem~\ref{theorem-braid}]
\label{theorem-braid-intro}
Let $\beta=\Delta^{2j} \s_1^{p_1} \s_2^{-q_1} \cdots \s_1^{p_k} \s_2^{-q_k}$ be a pseudo-Anosov $3$-braid with dilatation $\lambda$. 
Let $\ell= p_1 + q_1 + \cdots + p_k + q_k$ and $s$ be the number admitting a purely periodic continued fraction expansion 
$[ \ol{q_k: p_k, q_{k-1}, p_{k-1}, \cdots, q_1, p_1}].$ 

\begin{enumerate}
\item[(1)]
The measured train track 
$(\omega_0, \nu_0)$ with $\nu_0=\vs$ (Figure~\ref{fig_switch-omega}-(3))  
is suited to the stable measured lamination of $\beta$. 

\item[(2)] 
Let $A = L^{q_k} R^{p_k} \cdots L^{q_1} R^{p_1}$ where $L= \left(\begin{smallmatrix}1 & 0 \\ 1 & 1 \end{smallmatrix}\right)$ and 
$R= \left(\begin{smallmatrix}1 & 1 \\0 & 1\end{smallmatrix}\right).$ Then $A \vs = \lambda \vs$.

\item[(3)] 
Starting with the measured train track $(\omega_0, \nu_0)$, 
the first  $\ell+1$ terms 
\begin{equation}\label{eq:Agol_cycle}
(\omega_0, \nu_0)
\ls^{q_k} \ 
\rs^{p_k} \ 
\cdots \ 
\ls^{q_1} \ 
\rs^{p_1}
(\omega_\ell, \nu_\ell) 
\end{equation}
of  the maximal splitting sequence forms a length $\ell$ Agol cycle of $\beta$. 
\end{enumerate}
\end{theorem}

In other words, $(\omega_0, \nu_0)$ admits $q_k$ left maximal splittings consecutively. 
Then the resulting measured train track $(\omega_{q_k}, \nu_{q_k})$
admits $p_k$ right maximal splittings consecutively. 
After repeating this $k-1$ more times we obtain the measured train track $(\omega_\ell, \nu_\ell)$. 


\begin{figure}[htbp]
\centering
\includegraphics[height=3.4cm]{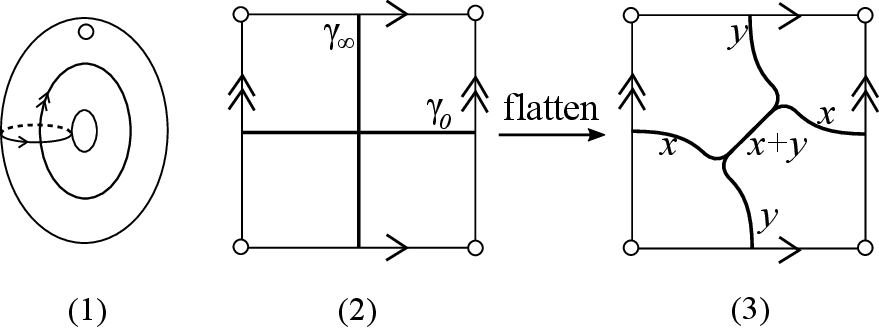}
\caption{(1) Once-punctured torus $\Sigma_{1,1}$. 
(2) Simple closed curves of slope $0$ and $\infty$. 
(3) The measured train track $(\tau_0, \vxy)$ in $\Sigma_{1,1}$.}
\label{fig_base}
\end{figure}

Here is a geometric description of the above train tracks $\omega_i$ $(i=0, \cdots, \ell)$: 
Each $(\omega_i, \nu_i)$ can be written by 
$(\omega_i, \nu_i)= (\omega_i, {K_i}^{-1} \vs)$ 
where $K_i=\abcd \in \SL(2; \Z)$. 
In the once-punctured torus $\Sigma_{1,1}$ we consider simple closed curves $\gamma_{\frac{b}{a}}$ of slope $\frac{b}{a}$ and $\gamma_{\frac{d}{c}}$ of slope $\frac{d}{c}$ that intersect at a single point. 
Flattening the obtuse angles at the intersection (Figure~\ref{fig_flatten}-(2),(3)) gives a train track $\tau_i$ in $\Sigma_{1,1}$. 
Two copies of $\tau_i$ nicely placed in $\Sigma_ {0,4}$  gives the train track $\omega_i$. 
For example in the case $i=0$ we have 
$(\omega_0, \nu_0)= (\omega_0, {K_0}^{-1} \vs)$ with 
$K_0= \left(\begin{smallmatrix}1 & 0 \\ 0 & 1 \end{smallmatrix}\right) $. 
Figure~\ref{fig_base}-(3) illustrates $\tau_0$ and the two copies of $\tau_0$ give the train track $\omega_0$ as in Figure~\ref{fig_switch-omega}-(3).

As a corollary of Theorem~\ref{theorem-braid-intro} we have the following. 

\noindent
{\bf Corollary~\ref{cor:A}.} {\em 
Pseudo-Anosov $3$-braids $\beta$ and $\beta'$ are conjugate in $B_3/Z(B_3)$ where $Z(B_3)$ is the center of $B_3$ 
if and only if their Agol cycles are equivalent. }

On the other hand, the equivalence class of an Agol cycle is not a complete conjugacy invariant of a pseudo-Anosov element of 
$\mathrm{MCG}(\Sigma_{1,1})$ (Proposition~\ref{prop_trace-negative}).

Given a pseudo-Anosov map $\phi:\Sigma\to\Sigma$, let $\Sigma^\circ \subset \Sigma$ denote the surface obtained  from $\Sigma$ by removing all the singular points of the stable/unstable foliations for $\phi$, and let $\phi^\circ : \Sigma^\circ \to  \Sigma^\circ$ be the restriction of $\phi$. 
Agol uses an Agol cycle of $\phi$ to give a  veering  ideal triangulation of the mapping torus of $\phi^\circ$ \cite{Agol11}. 
Every maximal splitting in the Agol cycle (\ref{eq:Agol_cycle}) takes place in two large branches for each maximal splitting, which yields two tetrahedra in Agol's construction.
Thus, for the $3$-braid case, Theorem~\ref{theorem-braid-intro} states that twice the Agol cycle length is exactly the number of tetrahedra in the Agol's triangulation.

To prove Theorem~\ref{theorem-braid-intro}, 
we give a thorough 
description of an Agol cycle of every 
pseudo-Anosov map on $\Sigma_{1,1}$  up to the hyperelliptic involution (Theorem~\ref{theorem-2}). 
It is known that the veering ideal triangulation of the 
mapping torus obtained from the Agol cycle 
is the canonical triangulation constructed in Lackenby \cite{Lackenby03} and Gu\'{e}ritaud \cite{Gueritaud06}.

%
%

\begin{question}[Margalit \cite{Margalit16}] 
For a fixed surface, what are the possible lengths of Agol cycles? 
How does the length of the Agol cycle related to other invariants? 
\end{question}

Theorem~\ref{theorem-braid-intro} implies that 
any integer greater than $1$ can be realized as the Agol cycle length of some pseudo-Anosov $3$-braid. 
Theorem~\ref{theorem-braid-intro}   partially answers the second question by Margalit as follows.  

\noindent
{\bf Theorem~\ref{thm:Garside}.} {\em 
For every pseudo-Anosov $3$-braid $\beta$, 
the Agol cycle length of $\beta$, the Garside canonical length of any element in the super summit set ${\rm SSS}(\beta)$ are the same.}

\begin{question}\label{question:length1}
Is there a pseudo-Anosov map  whose Agol cycle length is $1$?
\end{question} 

We note that a length $1$ Agol cycle does not necessarily mean that the induced veering triangulation of the mapping torus consists of one ideal tetrahedron, since a maximal splitting may contain multiple splittings simultaneously.

The paper is organized as follows. 
In Section~\ref{section_Preliminaries}, we recall basic definitions and facts regarding measured train tracks and laminations. 
In Section~\ref{sec:torus}, we introduce train tracks in  $\Sigma_{1,1}$ associated to Farey  intervals following Mosher \cite{Mosher03,Mosher2003}. 
Then we study Agol cycles of pseudo-Anosov maps on $\Sigma_{1,1}$. 
In Section~\ref{sec:applications}, we prove Theorem~\ref{theorem-braid-intro}. 
We also discuss a relation between Garside canonical lengths and Agol cycle lengths for $3$-braids and prove Theorem~\ref{thm:Garside}.

\section{Preliminaries}
\label{section_Preliminaries}

The mapping class group $\MCG(\Sigma)$ of a surface $\Sigma= \Sigma_{g,n}$ is the group of isotopy classes of orientation preserving homeomorphisms of $\Sigma$ which preserve the punctures setwise. 
For simplicity, we may not distinguish between a homeomorphism $\phi:\Sigma \to \Sigma$ and its mapping class $[\phi] \in \MCG(\Sigma)$. 

Measured train tracks are useful tools to encode measured laminations. 
Measured train tracks 
$(\tau, \mu)$, $(\tau', \mu')$ in $\Sigma$ are {\em equal} (and write $(\tau, \mu) = (\tau', \mu')$)
if there exists a diffeomorphism $f: \Sigma \rightarrow \Sigma$  isotopic to the identity map on $\Sigma$ such that 
$f(\tau, \mu)= (\tau', \mu')$.

Measured train tracks $(\tau, \mu)$, $(\tau', \mu')$ in $\Sigma$ are {\em equivalent} if they are related to each other by a sequence of splittings, foldings, shiftings (Definition~\ref{definition:3operations}) and isotopies.  
Equivalence classes of measured train tracks are in one-to-one correspondence with measured laminations \cite[Theorem~2.8.5]{PennerHarer92}. 
For example, all the five measured train tracks in $\Sigma_{0,4}$ in Figure~\ref{fig_five_tracks} are equivalent for any $s>0$.

\begin{figure}[htbp]
\begin{center}
\includegraphics[height=1.8cm]{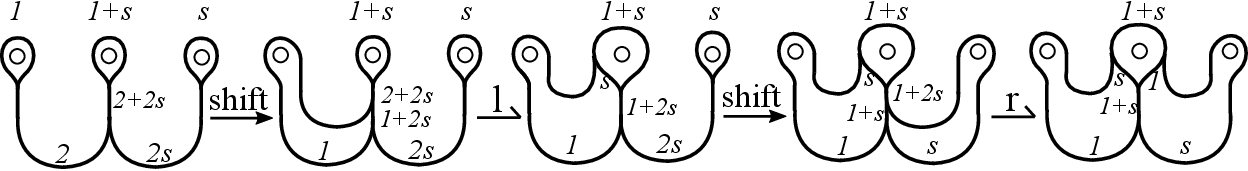}
\caption{
The last measured train track  is $(\omega_0, \left(\begin{smallmatrix} 1 \\ s \end{smallmatrix}\right))$ in Figure~\ref{fig_switch-omega}-(2).}
\label{fig_five_tracks}
\end{center}
\end{figure}

We adopt the following conventions. 
\begin{enumerate}
\item 
When we regard a maximal splitting $(\tau, \mu)\rightharpoonup (\tau', \mu')$ 
as an operation on the measured train track 
we may write
$ (\tau', \mu')=\  \rightharpoonup (\tau, \mu)$.  

\item 
Similarly, using the operator notation we may write $n$ consecutive left (resp. right) maximal splittings $(\tau, \mu) \ls^n (\tau_n, \mu_n)$ (resp. $(\tau, \mu) \rs^n (\tau_n, \mu_n)$) as $$(\tau_n, \mu_n) = \  \ls^n (\tau, \mu), \hspace{5mm} 
(\mbox{resp. } (\tau_n, \mu_n) =\   \rs^n (\tau, \mu)).$$
\item 
We may also write a finite sequence $(\tau, \mu) \ls^n (\tau_n, \mu_n) \rs^m (\tau_{n+m}, \mu_{n+m})$ as $$ (\tau_{n+m}, \mu_{n+m}) = \ \rs^m  \circ  \ls^n (\tau, \mu);$$
that is, first apply $ \ls^n$ to $(\tau, \mu)$ then next apply $\rs^m$ to obtain $ (\tau_{n+m}, \mu_{n+m})$. 
\end{enumerate}

The next lemma states that the operation $\rightharpoonup$ and a homeomorphism $\phi: \Sigma \rightarrow \Sigma$ commute 
on measured train tracks in $\Sigma$. 

\begin{lemma}\label{lemma:0A}
Let $(\tau, \mu)$ be a measured train track in $\Sigma$. 
Let $\phi: \Sigma \rightarrow \Sigma$ be an orientation-preserving homeomorphism. 
If $(\tau, \mu)$ admits consecutive $n$ maximal left splittings, then we have 
$$(\phi  \circ \ls^n) (\tau, \mu) = (\ls^n \circ \phi)  (\tau ,\mu).$$ 
A parallel statement holds for right splittings. 
\end{lemma}

\begin{proof}
Since a left splitting operation is supported in a small disk neighborhood of the large branch, they commute with any homeomorphism $\phi:\Sigma\to\Sigma$.
This gives $(\phi  \circ \ls) (\tau, \mu) = (\ls \circ \phi)  (\tau ,\mu)$.
Repeating this for $n$ times, we obtain $(\phi  \circ \ls^n) (\tau, \mu) = (\ls^n \circ \phi)  (\tau ,\mu).$
\end{proof}

As a corollary of Lemma~\ref{lemma:0A} we have the following. 
\begin{lemma}\label{lemma_comb_eq}
Let $(\tau_n, \mu_n) \rh (\tau_{n+1}, \mu_{n+1})  \rh\cdots$ and 
$(\tau'_{n'}, \mu'_{n'}) \rh (\tau'_{n'+1}, \mu'_{n'+1})  \rh\cdots$  be maximal splitting sequences. 
If there exist an orientation-preserving diffeomorphism $h:\Sigma\to\Sigma$, integers $p\geq n$, $q\geq n'$, and a positive number $c$ such that $h(\tau_p, \mu_p)=(\tau'_q, \mu'_q)$ then
$h(\tau_{p+i}, \mu_{p+i})=(\tau'_{q+i}, \mu'_{q+i})$ for all $i\geq 0$.
\end{lemma}

\begin{definition}
\label{definition:suited}
Let $(\mathcal{L}, \nu)$ be a measured lamination in $\Sigma$, 
and let $(\tau, \mu)$ be a measured train track in $\Sigma$. 
Then $(\mathcal{L},\nu)$ is {\em suited} to 
$(\tau, \mu)$,   
and we also say that $(\tau, \mu)$ is {\em suited} to $(\mathcal{L},\nu)$, 
if there exists a differentiable map 
$f: \Sigma \rightarrow \Sigma$ 
homotopic to the identity map on $\Sigma$ with the following conditions: 
\begin{itemize}
\item 
$f(\mathcal{L}) = \tau$. 

\item 
$f$ is non-singular on the tangent spaces to the leaves of $\mathcal{L}$. 

\item 
If $p$ is an interior point of a branch $e$ of $\tau$ 
then $\nu(f^{-1}(p)) = \mu(e)$. 
\end{itemize}
\end{definition}

\section{Once-punctured torus}\label{sec:torus}

\subsection{Mapping class group of once-punctured torus via $\SL(2;\Z)$}

The quotient space $\R^2 / \Z^2$ gives the torus, $\T = S^1 \times S^1$. 
If no confusion occurs, under the quotient map (or a covering map) 
$$Q:\R^2 \to \R^2/\Z^2 = \T$$ 
the image of $(x, y)\in \R^2$ will be denoted by the same $(x, y)$. 
We may think the torus is obtained by the square $[0,1]\times [0,1]$ whose parallel boundary edges are identified.

The special linear group $\SL(2;\Z)$ is generated by 
$$L= \left(\begin{array}{cc}1 & 0 \\1 & 1\end{array}\right) \hspace{5mm} \mbox{and} \hspace{5mm} 
R= \left(\begin{array}{cc}1 & 1 \\0 & 1\end{array}\right).$$
For  $A\in \SL(2; \Z)$, the linear map 
$A: \R^2 \to \R^2$; $\left(\begin{smallmatrix} x \\ y\end{smallmatrix}\right) \mapsto A\left(\begin{smallmatrix} x \\ y\end{smallmatrix}\right)$
induces a well defined  homeomorphism  $\f_A: \T \to \T$. 
Note that $\f_A$ fixes the point $(0,0) $. 
The restriction of $\f_A$ to the once-punctured torus 
$\Sigma_{1, 1}=\T \setminus \{(0,0)\} = ({\Bbb R}^2 \setminus {\Bbb Z}^2)/  {\Bbb Z}^2$  
yields a homeomorphism, denoted by 
$$f_A: \Sigma_{1,1} \to \Sigma_{1,1}.$$
We have $Q \circ A \vxy = f_A \circ Q \vxy$ and the isomorphism 
$\SL(2; \Z) \to  \MCG(\Sigma_{1,1})$ 
which takes $A$ to $f_A$ gives a groups isomorphism.

Observe that  $f_L$ induced by $L$ is the left-handed  Dehn twist about a simple closed curve with the slope $\infty$ and 
$f_R$ induced by $R$ is the right-handed Dehn twist about a simple closed curve with the slope $0$. 
(cf. Figure~\ref{fig_lrmap} for $f_L$ and $f_R$.)

Let $\tr(A)$ denote the trace of $A \in \SL(2;\Z)$. 
We have
$|\tr(A)| >2$ if and only if 
the induced map $\f_A: \T \to \T$ is Anosov \cite[Section~13.1]{FarbMargalit12}.  
A parallel statement for the punctured torus $\Sigma_{1,1}$ is that 
$|\tr(A)| >2$ if and only if 
the induced map $f_A: \Sigma_{1,1} \to \Sigma_{1,1}$ is pseudo-Anosov. 
The condition $\tr(A) > 2$ is equivalent to that 
$A $ possesses distinct eigenvalues $\lambda>1$ and $0 < \lambda^{-1} < 1$. 
We call $\lambda$ the {\em expanding eigenvalue} of $A$. Since the pseudo-Anosov map $f_A$ is restriction of the Anosov map $\f_A$, their dilatations are the same and equal to the expanding eigenvalue $\lambda$. 
Let $\vs$ be an eigenvector with respect to $\lambda$; that is, $A\vs=\lambda\vs$. 
We call $s$ the {\em slope} of the eigenvectors with respect to $\lambda$ for $A$. 

The following proposition is  well-known (cf. \cite[Proposition~2.1]{Gueritaud06}). 

\begin{proposition}
\label{prop:conjugacy-class-SL2Z}
Let $A \in \SL(2;\Z)$ with $\tr(A) > 2$. 
Then $A$ is conjugate to $L^{q_k} R^{p_k} \cdots L^{q_1} R^{p_1}$
for some positive integers $p_1, q_1, \cdots, p_k, q_k$ and $k$.   
Moreover, the ordered pairs $(p_1, q_1), \cdots, (p_k, q_k)$ are unique up to cyclic permutation.
\end{proposition}

\subsection{Mosher's train track}\label{sec:2-1}
In the rest of the paper, we assume that $a, b, c, d$ are nonnegative integers with $ad-bc=1.$
In other words, $\left(\begin{smallmatrix}a & c \\ b & d \end{smallmatrix}\right) \in \SL^+(2; \Z)$, where $\SL^+(2; \Z)$ is the monoid generated by $L$ and $R$. 
This condition is equivalent to that $\frac{b}{a}$ and $\frac{d}{c}$ are joined by an arc in the {\em Farey diagram}  
(see \cite[Section~1.1]{Hatcher22} for the Farey diagram) and $\frac{0}{1} \leq \frac{b}{a}<\frac{d}{c}\leq \frac{1}{0}$. 
We call the interval $[\frac{b}{a}, \frac{d}{c}]$ a {\em Farey interval}. 
We say that $[\frac{b}{a}, \frac{d}{c}]$ is the {\em corresponding Farey interval} to the matrix $\left(\begin{smallmatrix}a & c \\ b & d \end{smallmatrix}\right)$. 
We also say that $\left(\begin{smallmatrix}a & c \\ b & d \end{smallmatrix}\right)$ is the  {\em corresponding matrix} to the Farey interval $[\frac{b}{a}, \frac{d}{c}]$.

In \cite{Mosher03} and Sections~1.3 and 10.1 of \cite{Mosher2003} Mosher defined a train track $\tau_{[\frac{b}{a}, \frac{d}{c}]}$ in the torus $\T$ 
and showed an 
intriguing relation between the train track and a continued fraction expansion. 
In this paper we study Mosher's train track in the subspace $\Sigma_{1, 1} \subset \T$. 

We now  define a train track $\tau_{[\frac{b}{a}, \frac{d}{c}]}$ in $\Sigma_ {1,1}$ inductively. 
Our definition is different from Mosher's original one but in Proposition~\ref{propA1} we will see ours coincides with Mosher's.

Let $\tau_{[\frac{0}{1}, \frac{1}{0}]}:= \tau_0$ be a train track  
with measure  $\mu= \left(\begin{smallmatrix} x \\ y \end{smallmatrix}\right)$ in $\Sigma_{1,1}$ 
as shown in Figure~\ref{fig_base}-(3).  
The component of $\Sigma_{1,1} \setminus \tau_0$ is a once-punctured bigon.
The train track $\tau_0$ consists of three branches, one large branch and two small branches. 
The small branches from South and West merge to form a large branch then it separates into two small branches going North and East. 
We call $\tau_0$ the {\em base train track}. 
The vector $\mu= \left(\begin{smallmatrix} x \\ y \end{smallmatrix}\right)$ represents the weight $x$ 
of the horizontal small branch and the weight $y$ of the vertical small branch.  
From now on, the notation $\tau_0$ is exclusively used for the base train track.


Assume that we have defined a train track $\tau_{[\frac{b}{a}, \frac{d}{c}]}$ consisting of one large and two small branches. 
Note that $\left(\begin{smallmatrix}a & c \\ b & d \end{smallmatrix}\right) L=
 \left(\begin{smallmatrix}a & c \\ b & d \end{smallmatrix}\right)  \left(\begin{smallmatrix}1 & 0 \\ 1 & 1 \end{smallmatrix}\right) =
\left(\begin{smallmatrix}a+c & c \\ b+d & d \end{smallmatrix}\right) \in \SL^+(2; \Z)$.
In the Farey diagram, the corresponding Farey interval $[\frac{b+d}{a+c}, \frac{d}{c}]$ is the {\em right} half of the original Farey interval $[\frac{b}{a}, \frac{d}{c}]$.
We write 
$$
\left[\frac{b}{a}, \frac{d}{c}\right] \stackrel{\mbox{\tiny r-half}}{\supset} \left[\frac{b+d}{a+c}, \frac{d}{c}\right].
$$
We define the train track $\tau_{[\frac{b+d}{a+c}, \frac{d}{c}]}$ 
as a result of the left splitting of $\tau_{[\frac{b}{a}, \frac{d}{c}]}$ at the unique large branch. 
Although this is a topological operation (the measure is forgotten) 
abusing the left maximal splitting symbol, $\ls$, on measured train tracks, we may write
\begin{equation}\label{eq:left-split}
\tau_{[\frac{b}{a}, \frac{d}{c}]} \ls \tau_{[\frac{b+d}{a+c}, \frac{d}{c}]}.
\end{equation}

Similarly, the corresponding Farey interval of the matrix 
$\left(\begin{smallmatrix}a & c \\ b & d \end{smallmatrix}\right) R = \left(\begin{smallmatrix}a & a+c \\ b & b+d \end{smallmatrix}\right) \in \SL^+(2; \Z)$ 
is $[\frac{b}{a}, \frac{b+d}{a+c}]$ 
which is the {\em left} half of the original Farey interval $[\frac{b}{a}, \frac{d}{c}]$.
We  write 
$$\left[\frac{b}{a}, \frac{d}{c} \right] \stackrel{\mbox{\tiny l-half}}{\supset} \left[\frac{b}{a}, \frac{b+d}{a+c} \right].$$
We define the train track $\tau_{[\frac{b}{a}, \frac{b+d}{a+c}]}$ 
as a result of the right splitting of $\tau_{[\frac{b}{a}, \frac{d}{c}]}$.
Again, abusing the right maximal splitting symbol, $\rs$, we may write 
\begin{equation}\label{eq:right-split}
\tau_{[\frac{b}{a}, \frac{d}{c}]} \rs \tau_{[\frac{b}{a}, \frac{b+d}{a+c}]}.
\end{equation}

The both new train tracks $\tau_{[\frac{b+d}{a+c}, \frac{d}{c}]}$ and $\tau_{[\frac{b}{a}, \frac{b+d}{a+c}]}$ consist of three branches, 
one large and two small branches. 

For every Farey interval $[\frac{b}{a}, \frac{d}{c}]$, one can find a unique finite nested sequence of Farey intervals 
 starting from  $[\frac{0}{1}, \frac{1}{0}]$ and choosing the left/right half of it. 
Therefore, the train track $\tau_{[\frac{b}{a}, \frac{d}{c}]}$ is well-defined.

\begin{example}
\label{example:splitting-sequence}
The Farey interval $[\frac{10}{7}, \frac{3}{2}]$ is uniquely obtained as follows:
$$
\left[\frac{0}{1}, \frac{1}{0}\right] \stackrel{\mbox{\tiny r-half}}{\supset} 
\left[\frac{1}{1}, \frac{1}{0}\right] \stackrel{\mbox{\tiny l-half}}{\supset} 
\left[\frac{1}{1}, \frac{2}{1}\right] \stackrel{\mbox{\tiny l-half}}{\supset} 
\left[\frac{1}{1}, \frac{3}{2}\right] \stackrel{\mbox{\tiny r-half}}{\supset}  
\left[\frac{4}{3}, \frac{3}{2}\right] \stackrel{\mbox{\tiny r-half}}{\supset}  
\left[\frac{7}{5}, \frac{3}{2}\right] \stackrel{\mbox{\tiny r-half}}{\supset}  
\left[\frac{10}{7}, \frac{3}{2}\right].
$$
In terms of the corresponding matrices, we get 
$$
\left(\begin{smallmatrix} 7 & 2 \\ 10 & 3 \end{smallmatrix}\right)
= \left(\begin{smallmatrix} 1 & 0 \\ 0 & 1 \end{smallmatrix}\right) LRRLLL, 
$$
where $ \left(\begin{smallmatrix} 1 & 0 \\ 0 & 1 \end{smallmatrix}\right)$ 
is the corresponding matrix of the 
Farey interval $[\frac{0}{1}, \frac{1}{0}]$. 
The rule is to convert $\stackrel{\mbox{\tiny l-half}}{\supset}  / \stackrel{\mbox{\tiny r-half}}{\supset}$ into the matrix $R/L$ and multiply it from the right. 
Here we remark that l-half becomes $R$, and r-half becomes $L$. 
Thus, the train track $\tau_{[\frac{10}{7}, \frac{3}{2}]}$ is defined as a result of the following consecutive splittings:
$$\tz  \ls  \tau_{[\frac{1}{1}, \frac{1}{0}]}  \rs   \tau_{[\frac{1}{1}, \frac{2}{1}]}   \rs  \tau_{[\frac{1}{1}, \frac{3}{2}]}   \ls  \tau_{[\frac{4}{3}, \frac{3}{2}]}  \ls 
\tau_{[\frac{7}{5}, \frac{3}{2}]} \ls\ \tau_{[\frac{10}{7}, \frac{3}{2}]},$$
which is also written as $ \tz\  \ls \  \rs^{2} \  \ls^{3}   \tau_{[\frac{10}{7}, \frac{3}{2}]}$.
\end{example}

Recall that 
$f_ L, f_R: \Sigma_{1,1} \rightarrow \Sigma_{1,1}$ are induced by  $L, R \in  \SL(2; \Z)$ respectively. 

\begin{lemma}\label{lem:LandR}
We have 
$\tau_0  \ls \tau_{[\frac{1}{1}, \frac{1}{0}]}=f_L(\tau_0)$ 
and 
$\tau_0  \rs \tau_{[\frac{0}{1}, \frac{1}{1}]}=f_R(\tau_0)$. 
In other words 
$f_L(\tau_0) =\  \ls (\tau_0)$ and 
$f_R(\tau_0) = \ \rs (\tau_0)$. 
\end{lemma}

\begin{proof}
The left (resp. right) splitting of $\tau_0=\tz$ yields 
$\tau_{[\frac{1}{1}, \frac{1}{0}]}$ (resp.  $\tau_{[\frac{0}{1}, \frac{1}{1}]}$) by (\ref{eq:left-split}) (resp. (\ref{eq:right-split})) 
and 
it is equal to 
$f_L(\tau_0)$ (resp. $f_R(\tau_0)$) by  Figure~\ref{fig_lrmap}. 
\begin{figure}[htbp]
\centering
\includegraphics[height=7.5cm]{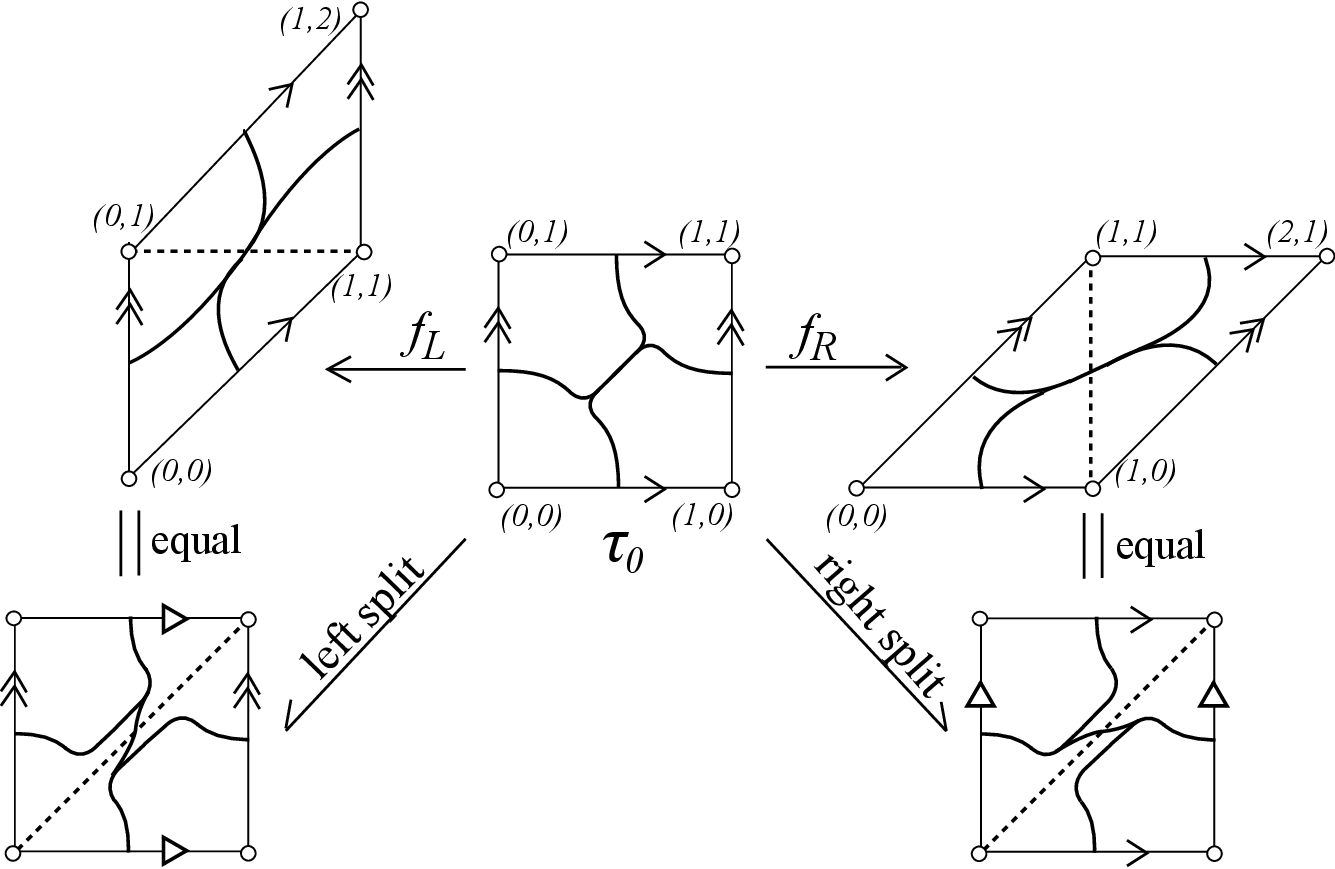}
\caption{Proof of Lemma~\ref{lem:LandR}: $\tau_0  \ls f_L(\tau_0)$ and $\tau_0  \rs f_R(\tau_0)$. }
\label{fig_lrmap}
\end{figure}
\end{proof}


%
%

\begin{proposition}\label{propA} 
For the base train track $\tau_0$ we have the following. 
\begin{enumerate}
\item 
For $p, q \ge 1$ we have 
$\tau_0 \ls^q \tau_{[\frac{q}{1}, \frac{1}{0}]}= f_L^q (\tau_0)$ and 
$\tau_0 \rs^p \tau_{[\frac{0}{1}, \frac{1}{p}]}= f_R^p (\tau_0)$. 
In other words 
$\ls^q (\tau_0) = f_L^q (\tau_0)$ and 
$\rs^p (\tau_0)=f_R^p (\tau_0)$.

\item 
For positive integers $p_1, q_1, \cdots, p_k, q_k$ and $k$ we have 
$$\tau_0
\ls^{q_k} \ 
\rs^{p_k} \ 
\cdots \ 
\ls^{q_1} \ 
\rs^{p_1}
\tau_{\ell},
$$
where 
$\ell= p_1+ q_1+ \cdots+ p_k+ q_k$ and 
$\tau_{\ell}= (f_L^{q_k} \circ f_R^{p_k} \circ \cdots \circ f_L^{q_1}\circ f_R^{p_1})(\tau_0)$. 
In other words 
$$
(f_L^{q_k} \circ f_R^{p_k} \circ \cdots \circ f_L^{q_1}\circ f_R^{p_1})(\tau_0) =(\rs^{p_1} \circ \ls^{q_1} \circ \cdots \circ \rs^{p_k} \circ \ls^{q_k})(\tau_0) .
$$
\end{enumerate}
\end{proposition}


\begin{proof}
We first prove Statement (1). 
When $q=1$ we have $\tau_0 \ls f_L(\tau_0)$ by Lemma~\ref{lem:LandR}. 
Suppose that $q=2$. 
By Lemmas~\ref{lemma:0A} and \ref{lem:LandR} it follows that 
$$
 (\ls \circ \ls) (\tau_0) 
 =  (\ls \circ f_L)(\tau_0)
 = (f_L \circ \ls) (\tau_0)
 = (f_L \circ f_L)(\tau_0).
 $$
Thus $\tau_0 \ls^2 f_L^2 (\tau_0)$. 
Repeating this argument we have 
$\tau_0 \ls^q f_L^q (\tau_0)$. 
One can similarly prove $\tau_0 \rs^p f_R^p (\tau_0)$. 

For Statement (2), we first consider the case $k=1$. 
Statement (1) together  with Lemma~\ref{lemma:0A} implies that 
$$
(\rs^p \circ \ls^q) (\tau_0) 
=
(\rs^p \circ f_L^q) (\tau_0)
=
(f_L^q \circ \rs^p) (\tau_0) 
=(f_L^q \circ f_R^p) (\tau_0). 
$$
The case $k=1$ is done. 
We turn to the case $k =2$. 
By the above argument it follows that 
$$
(\ls^{q_1} \circ \rs^{p_2} \circ \ls^{q_2}) (\tau_0) 
= (\ls^{q_1} \circ ( f_L^{q_2} \circ f_R^{p_2})) (\tau_0)
= ((f_L^{q_2} \circ f_R^{p_2}) \circ \ls^{q_1}) (\tau_0)
= (f_L^{q_2} \circ f_R^{p_2} \circ f_L^{q_1}) (\tau_0). 
$$
Similarly we have 
$$
(\rs^{p_1} \circ \ls^{q_1} \circ \rs^{p_2} \circ \ls^{q_2}) (\tau_0) 
=
(f_L^{q_2} \circ f_R^{p_2} \circ f_L^{q_1} \circ f_R^{p_1}) (\tau_0).
$$
Case $k = 2$ is done. 
The proof for the case $k \ge 3$ is similar. 
\end{proof}

Let $\gamma_{\frac{b}{a}}$ be a simple closed curve in $\Sigma_{1,1}$ whose slope $\frac{b}{a}$. 
We require that $\gamma_{\frac{b}{a}}$ is setwise preserved by the hyperelliptic involution $f_{-\I}$ where $-\I= \left(\begin{smallmatrix} -1 & 0 \\ 0 & -1 \end{smallmatrix}\right)$. 
See Figure~\ref{fig_flatten}-(1),(2). 

\begin{proposition}
\label{propA1}
\begin{enumerate}
\item 
For $A= \abcd \in \SL^+(2; \Z)$ we have $\tau_{[\frac{b}{a}, \frac{d}{c}]} = f_A(\tau_0)$. 

\item 
Moreover, 
if $A=  \abcd \ne \left(\begin{smallmatrix}1 & 0 \\ 0 & 1 \end{smallmatrix}\right)$ 
the train track $\tau_{[\frac{b}{a}, \frac{d}{c}]}$ 
is formed by flattening the obtuse angles at the intersection of  $\gamma_{\frac{b}{a}}$ and $\gamma_{\frac{d}{c}}$ 
as in Figure~\ref{fig_flatten}-(2),(3). 
In particular, $\tau_{[\frac{b}{a}, \frac{d}{c}]}$ is preserved by the hyperelliptic involution $f_{-\I}$. 
\end{enumerate}
\end{proposition}

Statement (2) is how Mosher originally defined $\tau_{[\frac{b}{a}, \frac{d}{c}]}$ in \cite{Mosher03,Mosher2003}.

\begin{proof}
By definition of the train track $\tau_{[\frac{b}{a}, \frac{d}{c}]}$ as in (\ref{eq:left-split}) and (\ref{eq:right-split}), 
the argument in the proof of Proposition~\ref{propA} yields Statement (1).  
See also Example~\ref{example:splitting-sequence}.

Recall that the base train track $\tau_0$ is 
obtained by flattening North West (NW) and South East (SE) right angles at the intersection of 
the simple closed curves $\gamma_{\frac{0}{1}}$ and $\gamma_{\frac{1}{0}}$ of  slope $0$ and $\infty$ as in Figure~\ref{fig_base}. 
If $A=  \abcd \ne \left(\begin{smallmatrix}1 & 0 \\ 0 & 1 \end{smallmatrix}\right)$ we have
\begin{equation}
\label{equation_slope}
\tfrac{0}{1} \le  \tfrac{b}{a}<\tfrac{d}{c}< \tfrac{1}{0} \hspace{5mm} \mbox{or} \hspace{5mm}\tfrac{0}{1} < \tfrac{b}{a}<\tfrac{d}{c} \le \tfrac{1}{0}. 
\end{equation}
Consider the train tack $f_A(\tau_0)$, the image  of $\tau_0$ under $f_A: \Sigma_{1,1} \to \Sigma_{1,1}$. 
By (\ref{equation_slope}) we see that $A$ 
takes the square $[0,1]\times [0,1] \subset \R^2$ to a parallelogram in the first quadrant of $\R^2$.  
As a consequence the NW and SE right-angle corners at the intersection of $\gamma_{\frac{0}{1}}$ and $\gamma_{\frac{1}{0}}$ are mapped to the NW and SE obtuse angle corners at the intersection  of $f_A(\gamma_{\frac{0}{1}})= \gamma_{\frac{b}{a}}$ and 
$f_A(\gamma_{\frac{1}{0}})= \gamma_{\frac{d}{c}}$. 
This implies that $f_A(\tau_0)$ is the union of $f_A(\gamma_{\frac{0}{1}}) \cup f_A(\gamma_{\frac{1}{0}})$ with the obtuse angles at the intersection flatten. 
This proves Statement (2). 
\end{proof}

\begin{figure}[ht]
\centering
\includegraphics[height=3.2cm]{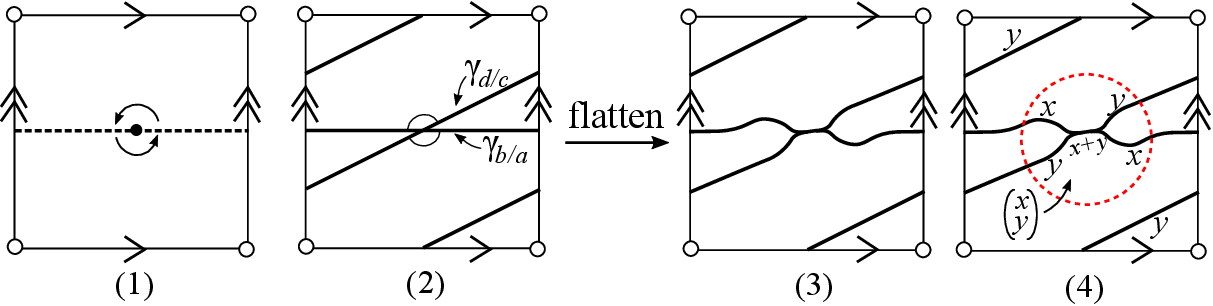}
\caption{(1) The hyperelliptic involution $f_{-\I}$. 
(2) The obtuse angles at the intersection point $\gamma_{\frac{b}{a}} \cap \gamma_{\frac{d}{c}}$. 
(3) Flattening gives $\tau_{[\frac{b}{a}, \frac{d}{c}]}$.
(4) Convention~\ref{def:measure convention} for a measure $\mu=\vxy$ on $\tau_{[\frac{b}{a}, \frac{d}{c}]}$. 
$\gamma_{\frac{b}{a}}= \gamma_{\frac{0}{1}}$ and $\gamma_{\frac{d}{c}}= \gamma_{\frac{1}{2}}$ in this figure.}
\label{fig_flatten}
\end{figure}

\subsection{Agol cycles of pseudo-Anosov maps on $\Sigma_{1,1}$} 

The next theorem describes Agol cycles of pseudo-Anosov maps on $\Sigma_ {1,1}$ 
induced by hyperbolic elements $A \in \SL(2; \Z)$ with $\tr(A) >2$.

\begin{theorem}\label{theorem-2}
Let $A = L^{q_k} R^{p_k} \cdots L^{q_1} R^{p_1} \in  \SL(2; \Z)$ 
where 
$p_1, q_1, \cdots, p_k, q_k$ and $k$ are positive integers. 
Let $\ell= p_1 + q_1 + \cdots + p_k + q_k$ and 
$s$ be the slope of the eigenvectors with respect to the expanding eigenvalue $\lambda>1$ of $A$. 
For the pseudo-Anosov map $f_A: \Sigma_{1,1} \rightarrow \Sigma_{1,1}$ with  dilatation $\lambda$, 
we have the following.  
\begin{enumerate}
\item[(1)]
The measured train track $(\tau_0, \mu_0)= (\tau_{[\frac{0}{1}, \frac{1}{0}]}, \vs)$ (Figure~\ref{fig_base}-(3)) is suited to 
the stable measured lamination of  $f_A$. 

\item[(2)] 
Starting with the measured train track 
$ (\tau_0, \mu_0)$, 
the first $\ell+1$ terms 
$$(\tau_0, \mu_0)
\ls^{q_k} \ 
\rs^{p_k} \ 
\cdots \ 
\ls^{q_1} \ 
\rs^{p_1}
(\tau_\ell, \mu_\ell)
$$
of the maximal splitting sequence satisfies 
$(\tau_{\ell}, \mu_{\ell}) = f_A(\tau_0, \lambda^{-1} \mu_0)$.
Thus, it forms a length $\ell$ Agol cycle of $f_A$.
Moreover, $\tau_\ell=\tau_{[\frac{b}{a}, \frac{d}{c}]}$, 
where $\left(\begin{smallmatrix}a & c \\ b & d \end{smallmatrix}\right) = A$. 
\end{enumerate}
\end{theorem}

Each entry of $A$ in Theorem~\ref{theorem-2} is positive; thus, $s>0$ by the Perron-Frobenius theorem.

\begin{convention}\label{def:measure convention}
Let $\tau= \tau_{[\frac{b}{a}, \frac{d}{c}]}$ be the train track in $\Sigma_{1,1}$ defined in Section~\ref{sec:2-1}. 
We fix a convention for a measure $\mu$ on $\tau_{[\frac{b}{a}, \frac{d}{c}]}$.
See Figure~\ref{fig_flatten}-(4). 
By Proposition~\ref{propA1}-(2)  
the train track $\tau_{[\frac{b}{a}, \frac{d}{c}]}$ is formed flattening the obtuse angles at the intersection of  $\gamma_{\frac{b}{a}}$ and $\gamma_{\frac{d}{c}}$. 
Recall that $\frac{0}{1} \le \frac{b}{a}< \frac{d}{c} \le \frac{1}{0}$. 
Let $x$ (resp. $y$) be the weight of the small branch of $\tau_{[\frac{b}{a}, \frac{d}{c}]}$ 
which was originally contained in $\gamma_{\frac{b}{a}}$ (resp. $\gamma_{\frac{d}{c}}$) before the flattening. 
The large branch has weight $x+y$ by the switch condition. 
The vector $\left(\begin{smallmatrix} x \\ y \end{smallmatrix} \right)$ represents the measure $\mu$ and we write 
$\mu= \left(\begin{smallmatrix} x \\ y \end{smallmatrix} \right)_{\tau}$ 
specifying the train track. When there is no confusion, we simply denote it by $\mu= \left(\begin{smallmatrix} x \\ y \end{smallmatrix} \right)$. 
Note that  $( \tau_{[\frac{0}{1}, \frac{1}{0}]}, \left(\begin{smallmatrix} x \\ y \end{smallmatrix} \right))$ in Figure~\ref{fig_base}-(3) aligns with this convention. 
\end{convention}

\begin{remark}
\label{remark:scc}
We allow $\mu= \left(\begin{smallmatrix} 1 \\ 0 \end{smallmatrix} \right)$ and $\left(\begin{smallmatrix} 0 \\ 1 \end{smallmatrix} \right)$ 
in Convention~\ref{def:measure convention}. 
In these cases, the measured train track $(\tau_{[\frac{b}{a}, \frac{d}{c}]}, \mu)$  yields the simple closed curves $\gamma_{\frac{b}{a}}$ and $\gamma_{\frac{d}{c}}$ respectively.  
\end{remark}

\begin{lemma}
\label{lemma:basic}
\begin{enumerate}
\item 
If $x < y$ then 
$(\tau_{[\frac{b}{a}, \frac{d}{c}]}, \mu = \left(\begin{smallmatrix} x \\ y \end{smallmatrix}\right)) \ls 
(\tau_{[\frac{b+d}{a+c}, \frac{d}{c}]}, L^{-1} \mu= \left(\begin{smallmatrix} x \\ y-x \end{smallmatrix}\right))$. 

\item 
If $x>y$ then 
$(\tau_{[\frac{b}{a}, \frac{d}{c}]}, \mu = \left(\begin{smallmatrix} x \\ y \end{smallmatrix}\right)) \rs 
(\tau_{[\frac{b}{a}, \frac{b+d}{a+c}]}, R^{-1} \mu= \left(\begin{smallmatrix} x-y \\ y \end{smallmatrix}\right))$. 
\end{enumerate}
\end{lemma}

\begin{proof}
We prove Statement (1). The proof of Statement (2) is similar. 
Recall that 
the train track $\tau_{[\frac{b}{a}, \frac{d}{c}]}$ consists of one large and two small branches. 
If $x< y$  the measured train track $(\tau_{[\frac{b}{a}, \frac{d}{c}]}, \mu )$ admits a left splitting 
$(\tau_{[\frac{b}{a}, \frac{d}{c}]}, \mu ) \ls (\tau', \mu')$. 
Then $\tau'= \tau_{[\frac{b+d}{a+c}, \frac{d}{c}]}$ by (\ref{eq:left-split}). 
To verify $\mu'= L^{-1} \mu$, 
we use Figure~\ref{fig_split_shift}-(1) substituting $z=y$ and $w=x$. 
Then by Convention~\ref{def:measure convention}  
we have $\mu'= L^{-1} \mu=  \left(\begin{smallmatrix} x \\ y-x \end{smallmatrix}\right)$. 
This completes the proof. 
\end{proof}


\begin{lemma}
\label{lemma-A}
For the measured train track $(\tau_{[\frac{b}{a}, \frac{d}{c}]}, \mu = \left(\begin{smallmatrix} x \\ y \end{smallmatrix}\right))$ 
we have the following. 

\begin{enumerate}
\item 
Suppose that $y=  qx +r$ with the quotient $q \in \N$ and the remainder $0\leq r< x$. 
Then $(\tau_{[\frac{b}{a}, \frac{d}{c}]}, \mu )$ admits $q$ left splittings consecutively and 
$$(\tau_{[\frac{b}{a}, \frac{d}{c}]}, \mu ) \ls^q 
(\tau', \mu'):= (\tau_{[\frac{b+ qd}{a+ qc}, \frac{d}{c}]}, L^{-q}\mu = \left(\begin{smallmatrix} x \\ r \end{smallmatrix}\right)).$$
Since $r \not> x$, $(\tau', \mu')$ cannot admit any more left splittings. 
Moreover, if the remainder $r \neq 0$ then $(\tau', \mu')$ falls into Case (2) below.

\item
Suppose that $x=  py + r$ with the quotient $p \in \N$ and the remainder $0\leq r< y$. 
Then  $(\tau_{[\frac{b}{a}, \frac{d}{c}]}, \mu )$ admits $p$ right splittings consecutively and 
$$(\tau_{[\frac{b}{a}, \frac{d}{c}]}, \mu ) \rs^p 
(\tau'', \mu''):=  (\tau_{[\frac{b}{a}, \frac{pb+d}{pa+c}]}, R^{-p} \mu  =\left(\begin{smallmatrix} r \\ y \end{smallmatrix}\right)). $$
Since $r \not> y$, $(\tau'', \mu'')$ cannot admit any more right splittings. 
Moreover, if the remainder $r \neq 0$ 
then $(\tau'', \mu'')$ falls into Case (1) above. 
\end{enumerate}
\end{lemma}

\begin{proof}
Applying Lemma~\ref{lemma:basic} repeatedly, we obtain the desired statements. 
\end{proof}

\begin{corollary}\label{cor:tau-consequtive}
For the measured train track $(\tau_0, \vxy)$ we have the following. 
\begin{enumerate}
\item
If $y>qx$ then  $(\tau_0, \vxy)$ admits $q$ left splittings consecutively  and 
$$(\tau_0, \vxy) \ls^q \ f_L^q(\tau_0,  L^{-q}\vxy).$$
\item
If $x>py$ then  $(\tau_0, \vxy)$ admits $p$ right splittings consecutively and
$$(\tau_0, \vxy) \rs^p \ f_R^p (\tau_0, R^{-p}\vxy).$$
\end{enumerate}
\end{corollary}

\begin{proof}
We prove Statement (1). Statement (2) follows similarly. 
Suppose that $y>qx$. 
By Proposition~\ref{propA}-(1) and Lemma~\ref{lemma-A}-(1), 
it follows that 
$(\tau_0,\vxy ) \ls^q  (f_L^q(\tau_0), L^{-q}\vxy)$. 
By Convention ~\ref{def:measure convention} 
it holds 
$(f_L^q(\tau_0), L^{-q}\vxy)= f_L^q(\tau_0, L^{-q}\vxy)$ and 
we obtain $(\tau_0,\vxy ) \ls^q f_L^q(\tau_0, L^{-q}\vxy) $.  
\end{proof}

Let $A = L^{q_k} R^{p_k} \cdots L^{q_1} R^{p_1} $ as in Theorem~\ref{theorem-2}. 
Let $s$ be the slope of the eigenvectors with respect to the expanding eigenvalue $\lambda>1$ of $A$, that is 
$A \left(\begin{smallmatrix} 1 \\ s \end{smallmatrix}\right) = \lambda \left(\begin{smallmatrix} 1 \\ s \end{smallmatrix}\right)$. 
Since $\lambda^2 - \tr(A) \lambda + 1=0$ and $\tr(A) >2$, 
the eigenvalues of $A$ are quadratic irrationals. 
This implies that $s$ is also a quadratic irrational.

Let us consider the 
infinite continued fraction expansion of $s$. 
$$
s= n_0+ \cfrac{1}{n_1 + 
           \cfrac{1}{\ddots+ 
           \cfrac{1}{n_{k}+ \cdots }}}
= [n_0:n_1, \cdots, n_k, \cdots]
$$
with $n_i \in {\Bbb Z}$ and $n_i >0$ for $i \ge 1$. 
By Lagrange's theorem  
the expansion is eventually periodic. 
i.e., there exists $t \ge 1$ with $n_i = n_{i+t}$ for all $i \gg 1$ \cite{Karpenkov22}. 
Hence the expansion of $s$ is of the form: 
$$s= [n_0: n_1, \cdots, n_{k-1}, m_0, m_1, \cdots, m_{t-1}, m_0, m_1, \cdots, m_{t-1}, \cdots\ ],$$ 
which is denoted by $[n_0: n_1, \cdots, n_{k-1}, \ol{m_0, m_1, \cdots, m_{t-1}}]$. 
We now claim that the expansion of $s$ is purely periodic.

\begin{proposition}\label{prop:s}
The  slope $s$ satisfies the following: 
\begin{enumerate}
\item
$s>1$.  
\item
$s$ admits a purely periodic continued fraction expansion 
$$s = [ \ol{q_k: p_k, q_{k-1}, p_{k-1}, \cdots, q_1, p_1}].$$
\end{enumerate}
\end{proposition}

\begin{proof}
Elementary computation shows that $s>0$ as follows. 
Let $A=\left(\begin{smallmatrix}a & b \\ c & d\end{smallmatrix}\right)$. 
Since $p_1, q_1, \cdots, p_k, q_k$ are positive integers $a, b, c, d$ are also positive integers. 
We note $s = (\lambda-a)/b$ and $\lambda -a = (-a+d+\sqrt{(a+d)^2 - 4})/2$. We see  $-a+d+\sqrt{(a+d)^2 - 4} >0$ because $\tr(A)=a+d>2$ and 
\begin{eqnarray*}
(d+\sqrt{(a+d)^2 - 4})^2 - a^2 
&=&  2d(a+d)-4 + 2d\sqrt{(a+d)^2 - 4}\\ 
&\geq& 2(a+d)-4 + 2\sqrt{(a+d)^2 - 4}\\ 
&\geq&  2\sqrt{(a+d)^2 - 4} >0.
\end{eqnarray*}
We conclude $s>0$. 
(This fact also follows from the Perron-Frobenius theorem.)

Denote 
\begin{eqnarray*}
\left(\begin{smallmatrix}
x_0 \\ y_0
\end{smallmatrix}\right)
&:= &
\left(\begin{smallmatrix}
1 \\ s
\end{smallmatrix}\right),
\\
\left(\begin{smallmatrix}
x_1 \\ y_0
\end{smallmatrix}\right)
&:= & R^{p_1}
\left(\begin{smallmatrix}
x_0 \\ y_0
\end{smallmatrix}\right) 
= 
\left(\begin{smallmatrix}
x_0 + p_1 y_0 \\ y_0
\end{smallmatrix}\right), 
\\
\left(\begin{smallmatrix}
x_1 \\ y_1
\end{smallmatrix}\right)
&:= & 
L^{q_1}  R^{p_1}
\left(\begin{smallmatrix}
x_0 \\ y_0
\end{smallmatrix}\right) 
=
L^{q_1}
\left(\begin{smallmatrix}
x_1 \\ y_0
\end{smallmatrix}\right) 
= 
\left(\begin{smallmatrix}
x_1 \\ q_1 x_1 + y_0
\end{smallmatrix}\right), 
\\
& \vdots &
\\
\left(\begin{smallmatrix}
x_k \\ y_{k-1}
\end{smallmatrix}\right)
&:= & 
R^{p_k} L^{q_{k-1}}\cdots L^{q_1}  R^{p_1}
\left(\begin{smallmatrix}
x_0 \\ y_0
\end{smallmatrix}\right) 
=
R^{p_k}
\left(\begin{smallmatrix}
x_{k-1} \\ y_{k-1}
\end{smallmatrix}\right) 
= 
\left(\begin{smallmatrix}
x_{k-1}+ p_k y_{k-1} \\ y_{k-1}
\end{smallmatrix}\right), 
\\
\left(\begin{smallmatrix}
x_k \\ y_k
\end{smallmatrix}\right)
&:= & 
L^{q_k}R^{p_k} \cdots L^{q_1}  R^{p_1}
\left(\begin{smallmatrix}
x_0 \\ y_0
\end{smallmatrix}\right) 
=
L^{q_k}
\left(\begin{smallmatrix}
x_k \\ y_{k-1}
\end{smallmatrix}\right) 
= 
\left(\begin{smallmatrix}
x_k \\ q_k x_k + y_{k-1}
\end{smallmatrix}\right). 
\end{eqnarray*}
Since exponents $p_1, q_1, \cdots, p_k, q_k$ are all positive integers we have two cases depending on (i) $1<s$ or (ii) $s<1$. 
\begin{enumerate}
\item[(i)]
$1= x_0  < s=y_0 < x_1 < y_1 < x_2 < y_2 < \cdots < y_{k-1} < x_k < y_k$.
\item[(ii)]
$s=y_0 < 1=x_0 < x_1 < y_1 < x_2 < y_2 < \cdots < y_{k-1} < x_k < y_k.$
\end{enumerate}
Since $A \vs=\lambda \vs$
we have 
$\left(\begin{smallmatrix}x_k \\ y_k\end{smallmatrix}\right) 
= \lambda 
\left(\begin{smallmatrix}x_0 \\ y_0\end{smallmatrix}\right)
= \lambda 
\left(\begin{smallmatrix}1 \\ s\end{smallmatrix}\right)
$.
For (ii) we have $x_k=\lambda > s \lambda = y_k,$ which is a contradiction. 
Thus, we conclude $s>1$.
Statement (1) is proved.

Reading the above computation from backward we obtain $2k$ set of division-with-remainder equations as appear in the Euclidean Algorithm: 
\begin{eqnarray*}
y_k &=& q_k x_k + y_{k-1}, \\
x_k &=& p_k y_{k-1} + x_{k-1},\\
& \vdots & \\
y_1 &=& q_1 x_1 + y_0, \\
x_1 &=&  p_1 y_0 + x_0.
\end{eqnarray*}
Put $s_i = \frac{1}{\lambda} y_i$ and 
$r_i = \frac{1}{\lambda} x_i.$
Since $\lambda 
\left(\begin{smallmatrix}1 \\ s\end{smallmatrix}\right)= \left(\begin{smallmatrix}x_k \\ y_k\end{smallmatrix}\right) 
= \left(\begin{smallmatrix}x_k \\ q_k x_k + y_{k-1} \end{smallmatrix}\right)
$, multiplying $\frac{1}{\lambda}$ to both sides of the above $2k$ equations gives the following new $2k$ division-with-remainder equations: 
\begin{eqnarray}
s= s_k &=& q_k \cdot 1 + s_{k-1}, \label{eq:s1} \\
1= r_k &=& p_k \cdot s_{k-1} + r_{k-1},\\
&\vdots& \nonumber
\\  
s_1 &=& q_1 r_1 + s_0, \\
r_1 &=& p_1 s_0 + r_0, \label{eq:s2k}
\end{eqnarray}
where 
$
0<r_0 < s_0 < \cdots < r_{k-1} < s_{k-1} < 1 < s.
$
We obtain
\begin{eqnarray*}
\left(\begin{smallmatrix} 1 \\ s_{k-1} \end{smallmatrix}\right)
&=& 
\left(\begin{smallmatrix} 1 & 0\\ -q_k & 1 \end{smallmatrix}\right)
\left(\begin{smallmatrix} 1 \\ s \end{smallmatrix}\right),
\\
\left(\begin{smallmatrix} r_{k-1} \\ s_{k-1} \end{smallmatrix}\right)
&=& 
\left(\begin{smallmatrix} 1 & -p_k\\ 0 & 1 \end{smallmatrix}\right)
\left(\begin{smallmatrix} 1 \\ s_{k-1} \end{smallmatrix}\right),
\\
&\vdots&
\\
\left(\begin{smallmatrix} r_1 \\ s_0 \end{smallmatrix}\right)
&=& 
\left(\begin{smallmatrix} 1 & 0\\ -q_1 & 1 \end{smallmatrix}\right)
\left(\begin{smallmatrix} r_1 \\ s_1 \end{smallmatrix}\right),
\\
\left(\begin{smallmatrix} r_0 \\ s_0 \end{smallmatrix}\right)
&=& 
\left(\begin{smallmatrix} 1 & -p_1\\ 0 & 1 \end{smallmatrix}\right)
\left(\begin{smallmatrix} r_1 \\ s_0 \end{smallmatrix}\right).
\end{eqnarray*}
Thus, 
$
\left(\begin{smallmatrix} r_0 \\ s_0\end{smallmatrix}\right) 
=R^{-p_1}L^{-q_1}\cdots R^{-p_k}L^{-q_k}
\left(\begin{smallmatrix} 1 \\ s \end{smallmatrix}\right)
=A^{-1}\vs 
=\lambda^{-1}
\left(\begin{smallmatrix} 1 \\ s \end{smallmatrix}\right).
$
In other words, $\frac{s_0}{r_0}=s$ and 
$s=[q_k: p_k, q_{k-1}, p_{k-1}, \cdots, q_1, p_1, \frac{s_0}{r_0}=s].$ 
This shows that $s$ admits the purely periodic continued fraction expansion 
$s = [\ol{q_k: p_k, q_{k-1}, p_{k-1}, \cdots, q_1, p_1}].
$
\end{proof}

Stable and unstable measured laminations of a pseudo-Anosov map $\phi: \Sigma \rightarrow \Sigma$ are unique up to scaling and isotopy. 
The measured train track $(\tau, \mu)$ in Theorem~\ref{thm:Agol} is assumed to be trivalent and suited to the stable measured lamination. 
We explain how to find  such $(\tau, \mu)$: 
The Bestvina-Handel algorithm \cite{BestvinaHandel95} 
determines the Nielsen-Thurston type of a homeomorphism $\phi: \Sigma \rightarrow \Sigma$. 
If $\phi$ is  pseudo-Anosov then 
the algorithm provides us a measured train track $(\tau', \mu')$ which is suited to the stable measured lamination $(\mathcal{L}, \nu)$ of $\phi$.  
If $\tau'$ is trivalent, then 
we set $(\tau,\mu):= (\tau', \mu')$. 
Otherwise, 
by {\em combing} branches near switches with degree greater than $3$ (see \cite[Figure~7]{HodgsonIssaSegerman16}) 
we can obtain a trivalent train track $\tau$  with a transverse measure $\mu$ 
so that $(\tau, \mu)$ is suited to  $(\mathcal{L}, \nu)$. 
See  \cite[Section~3]{HodgsonIssaSegerman16} for more details. 

We are ready to prove Theorem~\ref{theorem-2}.

\begin{proof}[Proof of Theorem~\ref{theorem-2}]

Let $A= L^{q_k}R^{p_k} \cdots L^{q_1} R^{p_1}$ as in Theorem~\ref{theorem-2}. 
One sees that the measured train track
$(\tau_0,  \vs)$ (up to scaling)  
is obtained by applying the Bestvina-Handel algorithm to $f_A$. 
Since $\tau_0$ is trivalent Statement (1) holds.

We reuse the computations (\ref{eq:s1})--(\ref{eq:s2k}). 
Since $s = q_k \cdot 1 + s_{k-1}$
applying Lemma~\ref{lemma-A}-(1) and Corollary~\ref{cor:tau-consequtive}-(1) gives  
$$(\tau_0, \vs) \ls^{q_k} (\tau_{q_k}, \left(\begin{smallmatrix}1 \\ s_{k-1} \end{smallmatrix}\right)) =(\tau_{q_k}, L^{-q_k} \vs) = f_L^{q_k}(\tau_0, L^{-q_k} \vs).$$
 Since $1 = p_k \cdot s_{k-1} + r_{k-1}$ 
 we can apply Lemma~\ref{lemma-A}-(2) to $f_L^{q_k}(\tau_0, L^{-q_k} \vs)$ 
 and we see that 
 $f_L^{q_k}(\tau_0, L^{-q_k} \vs)$ admits $p_k$ right splittings consecutively. 
 Then by Lemma \ref{lemma:0A} we obtain 
 $$(\rs^{p_k} \circ f_L^{q_k})(\tau_0, L^{-q_k} \vs) = (f_L^{q_k} \circ \rs^{p_k})(\tau_0, L^{-q_k} \vs) 
 = (f_L^{q_k} \circ f_R^{p_k}) (\tau_0, R^{-p_k} L^{-q_k} \vs). $$
 In other words 
 $f_L^{q_k}(\tau_0, L^{-q_k} \vs) \rs^{p_k}  f_L^{q_k} \circ f_R^{p_k} (\tau_0, R^{-p_k} L^{-q_k} \vs) =  
  f_L^{q_k} \circ f_R^{p_k}(\tau_0,  \left(\begin{smallmatrix}r_{k-1} \\ s_{k-1} \end{smallmatrix}\right))$.  
  After repeating this $k-1$ more times, we obtain 
  $$(\tau_0, \vs) \ls^{q_k} \ \rs^{p_k} \cdots \ls^{q_1} \ \rs^{p_1} (\tau_{\ell}, \mu_{\ell}),$$ 
  where 
$(\tau_{\ell}, \mu_{\ell}) = (f_L^{q_k} \circ f_R^{p_k} \circ \dots \circ f_L^{q_1} \circ f_R^{p_1}) (\tau_0, R^{-p_1} L^{-q_1} \dots   R^{-p_k} L^{-q_k}  \vs) 
= f_A (\tau_0, A^{-1} \vs) = f_A (\tau_0, \lambda^{-1} \vs)$, which gives a length $\ell$ Agol cycle. 
By Proposition~\ref{propA1}-(1) we have 
$ \tau_ {\ell}= f_A(\tau_0)= \tau_{[\frac{b}{a}, \frac{d}{c}]}$, 
where 
$\left(\begin{smallmatrix}a & c \\ b & d \end{smallmatrix}\right) =A$.   
This completes the proof.
\end{proof}


\begin{example}
\label{ex_lr_torus}
Let $A = LR = \left(\begin{smallmatrix}1 & 1 \\ 1 & 2 \end{smallmatrix}\right)$. 
The mapping torus of $f_A: \Sigma_{1,1} \rightarrow \Sigma_{1,1}$ is the figure eight knot complement. 
See \cite{Gabai83} for example.   
The slope of eigenvectors with respect to the expanding eigenvalue $\lambda= \tfrac{3+ \sqrt{5}}{2}$ of $A$ is the golden ratio $s =  \tfrac{1+ \sqrt{5}}{2}$ which admits the continued fractional expansion $[\bar{1}]$. 
By Theorem~\ref{theorem-2} and Lemma~\ref{lemma:basic}, 
$(\tau_0, \left(\begin{smallmatrix} 1 \\ s \end{smallmatrix}\right)) \ls 
(\tau_{[\frac{1}{1}, \frac{1}{0}]}, \left(\begin{smallmatrix} 1 \\ s-1 \end{smallmatrix}\right)) 
\rs 
(\tau_{[\frac{1}{1}, \frac{2}{1}]}, \left(\begin{smallmatrix} 2-s \\ s-1 \end{smallmatrix}\right)) 
= f_A (\tau_0, \lambda^{-1} \left(\begin{smallmatrix} 1 \\ s \end{smallmatrix}\right))$  
forms a length $2$ Agol cycle of  $f_A$. 
See Figure~\ref{fig_lr_torus}. 
\begin{figure}[htbp]
\centering
\includegraphics[height=3.2cm]{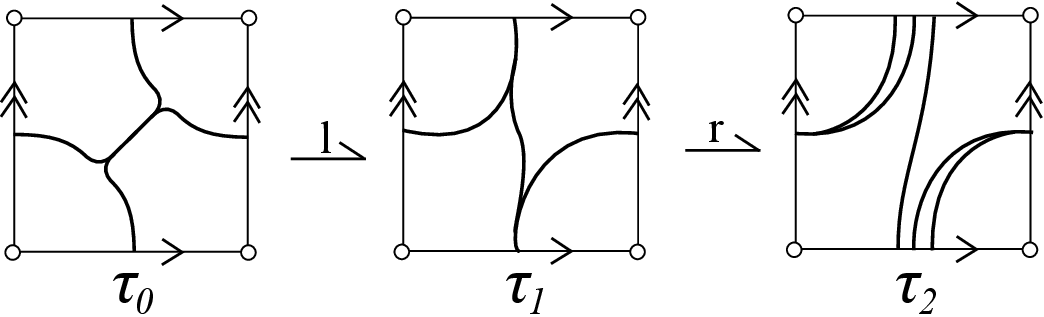}
\caption{Example~\ref{ex_lr_torus}: 
$\tau_0= \tau_{[\frac{0}{1}, \frac{1}{0}]} \ls \tau_1= \tau_{[\frac{1}{1}, \frac{1}{0}]} \rs \tau_2= \tau_{[\frac{1}{1}, \frac{2}{1}]}$.}
\label{fig_lr_torus} 
\end{figure}
\end{example}

Agol cycles of pseudo-Anosov maps induced by hyperbolic elements $A' \in  \SL(2;\Z)$ 
with $\tr(A')< -2$ have the following property. 

\begin{proposition}
\label{prop_trace-negative}
Let $A \in \SL^+(2;\Z)$ with $\tr(A) >2$. 
The pseudo-Anosov maps $f_A$ and $f_{-A}$ are not conjugate but they have equivalent Agol cycles. 
\end{proposition}

\begin{proof} 
We note that $f_{-A}$ is pseudo-Anosov  with dilatation equal to that of $f_A$.
Since $\tr(A) =-\tr(-A) \ne \tr(-A)$, $f_{-A}$ is not conjugate to $f_A$.  
Let 
\begin{equation}\label{eq:agol_cycle1}
(\tau_0, \mu_0)
\ls^{q_k} \ 
\rs^{p_k} \ 
\cdots \ 
\ls^{q_1} \ 
\rs^{p_1}
(\tau_\ell, \mu_\ell) = f_A(\tau_0, \lambda^{-1}\mu_0)
\end{equation}
be the Agol cycle of $f_A$ as in Theorem~\ref{theorem-2}. 
Since $(\tau_{[\frac{b}{a}, \frac{d}{c}]}, \vxy)$ is preserved by $f_{-\I}$ 
we have 
$$(\tau_\ell, \mu_\ell) = f_{-\I}(\tau_\ell, \mu_\ell) = f_{-\I} \circ f_A(\tau_0, \lambda^{-1}\mu_0) = f_{-A} (\tau_0, \lambda^{-1}\mu_0).$$ 
Thus  (\ref{eq:agol_cycle1}) 
is also an Agol cycle of $f_{-A}$. 
\end{proof}

Mosher briefly explain the following result in \cite{Mosher03}.

\begin{proposition}\label{prop:A}
Let $s=\frac{b}{a} \in \Q \cap (0, \infty)$ be a rational number with $\gcd(a, b)=1$ 
which admits a finite continued fraction expansion 
$s=[n_0: n_1, \cdots, n_k],$
where $n_0$ (resp. $n_i$ for $i= 1, \dots, k$) is a non-negative (resp. positive) integer.
Let $\ell=n_0 + n_1 + \dots + n_k$. 
Starting with the measured train track $(\tau_0, \left(\begin{smallmatrix} a\\ b \end{smallmatrix}\right))$ (see Figure~\ref{fig_base}-(3)), 
we have a maximal splitting sequence of length $\ell$ ending with the simple closed curve $\gamma_{\frac{b}{a}}$: 
\begin{itemize}
\item
$(\tau_0, \left(\begin{smallmatrix} a\\ b \end{smallmatrix}\right))
\ls^{n_0}\ \
\rs^{n_1}\ \
\cdots \ \
\ls^{n_k}\ \
(\tau_\ell, \left(\begin{smallmatrix} 1 \\ 0 \end{smallmatrix}\right)) =
\gamma_{\frac{b}{a}}$ 
when $k$ is even. 
More precisely 
$\tau_\ell=\tau_{[\frac{b}{a}, \frac{\delta}{\gamma}]}$ 
with 
$\left(\begin{smallmatrix} a & \gamma \\ b & \delta \end{smallmatrix}\right) = L^{n_0} R^{n_1} \cdots L^{n_k}$.

\item
$(\tau_0, \left(\begin{smallmatrix} a\\ b \end{smallmatrix}\right))
\ls^{n_0}\ \
\rs^{n_1}\ \
\cdots \ \
\rs^{n_k}\ \
(\tau_\ell, \left(\begin{smallmatrix} 0 \\ 1 \end{smallmatrix}\right)) =
\gamma_{\frac{b}{a}}$
when $k$ is odd.
More precisely 
$\tau_\ell =\tau_{[\frac{\beta}{\alpha}, \frac{b}{a}]}$ 
with $\left(\begin{smallmatrix} \alpha & a \\ \beta & b \end{smallmatrix}\right) = L^{n_0} R^{n_1} \cdots R^{n_k}$. 
\end{itemize}
If $n_0=0$, 
equivalently $ 0 < s=\frac{b}{a} \le 1$, 
the symbol
$\ls^{0}$ is removed from the above splitting sequence and the splitting starts with $\rs^{n_1}$. 
\end{proposition}

The last splittings $\ls^{n_k}$ and $\rs^{n_k}$ in Proposition~\ref{prop:A} are extended maximal splittings; namely, 
they are maximal splittings in the sense of Definition~\ref{definition:3operations}-(1) where $x=z$ and $y=w$. 
In the extended case, the results of left and right splittings are the same if the small branch with the weight $0$ is ignored. 
See Figure~\ref{fig_split_shift}-(1). 
For the last measured train track which gives $\gamma_{\frac{b}{a}}$, see Remark~\ref{remark:scc}. 
The proof of Proposition~\ref{prop:A} is omitted as it is similar to that of Theorem~\ref{theorem-2}.

\begin{example}
We apply  Proposition~\ref{prop:A} to $s = \frac{10}{7}= [1: 2,3]$. 
See also Example~\ref{example:splitting-sequence}. 
By Lemma~\ref{lemma-A} we have 
$$
(\tau_{[\frac{0}{1}, \frac{1}{0}]}, \left(\begin{smallmatrix} 7\\ 10 \end{smallmatrix}\right)) \ls 
( \tau_{[\frac{1}{1}, \frac{1}{0}]}, \left(\begin{smallmatrix} 7\\ 3 \end{smallmatrix}\right)) \rs^2
( \tau_{[\frac{1}{1}, \frac{3}{2}]} , \left(\begin{smallmatrix} 1\\ 3 \end{smallmatrix}\right)) \ls^3 
(\tau_{[\frac{10}{7}, \frac{3}{2}]}, \left(\begin{smallmatrix} 1\\ 0 \end{smallmatrix}\right))= \gamma_{\frac{10}{7}}.
$$
\end{example}

\section{Applications}\label{sec:applications}

\subsection{The braid groups and related mapping class groups}
\label{subsection_braid-mappingclass}

Let $B_n$ be the group of $n$-braids, which has the well-known presentation 
with generators $\sigma_1, \dots, \sigma_{n-1}$ and the relations 
$\sigma_i \sigma_j= \sigma_j \sigma_i$ if $|i-j|>1$ and 
$\sigma_j \sigma_i \sigma_j = \sigma_i \sigma_j \sigma_i$ if $|i-j|=1$.

Let $D_n$ be a disk with $n$ punctures $a_1, \dots,a_n$.  
We denote by $\MCG(D_n, \partial D)$, 
the mapping class group of  homeomorphisms which fix the boundary $\partial D$ of the $n$-punctured disk pointwise. 
There is a well-known isomorphism 
$B_n \rightarrow \MCG(D_n, \partial D)$ 
sending $\sigma_i$ to a positive half twist $h_i$ which interchanges the punctures $a_i$ and $a_{i+1}$ counterclockwise. 
Under the identification $B_n = \MCG(D_n, \partial D)$ 
a braid word $b_1 b_2 \in B_n$ (read from left to right) is identified 
with the composition $b_2 \circ b_1 \in \MCG(D_n, \partial D)$ (read from right to left) of 
mapping classes.

Capping the boundary of $D_n$ with a once-punctured disk yields an $n+1$-punctured sphere $\Sigma_{0,n+1}$. 
The $n+1$ punctures of $\Sigma_{0,n+1}$  come from the punctures $a_1, \dots, a_{n}$ of $D_n$ and the puncture, say $a_\infty$, of the once-punctured capping disk. 
Let 
$ \MCG(D_n, \partial D) \stackrel{Cap}{\to}  \MCG(\Sigma_{0,n+1}, \{a_\infty\})$ denote the induced homomorphism 
where $\MCG(\Sigma_{0,n+1}, \{a_\infty\})$ denotes the subgroup of $\MCG(\Sigma_{0, n+1})$ consisting of elements that fix the puncture $a_\infty$. 
The following sequence is exact:
$$1 \to \langle \Delta^2 \rangle \to B_n \simeq \MCG(D_n, \partial D) \stackrel{Cap}{\to}  \MCG(\Sigma_{0,n+1}, \{a_\infty\}) \to 1,
$$
where $\Delta   \in B_n$ is a positive half twist and the infinite cyclic group 
$\langle \Delta^2 \rangle$ generated by the positive full twist $\Delta^2$ gives the center of $B_n$. 
Composing with the injective inclusion $\MCG(\Sigma_{0,n+1}, \{a_\infty\}) \hookrightarrow \MCG(\Sigma_{0,n+1})$ we obtain a homomorphism (strictly speaking anti-homomorphism) 
$$\Gamma: B_n \to \MCG(\Sigma_{0,n+1}).$$
A braid $b\in B_n$ is {\em pseudo-Anosov} if 
$\Gamma(b)$ is represented by a pseudo-Anosov map. 
The dilatation $\lambda(b)>1$ of the pseudo-Anosov braid $b$ 
is defined to be the dilatation of $\Gamma(b)$.

Abusing the notation, 
the image of the braid element $\s_i \in B_n$ 
via  $\Gamma: B_n \to \MCG(\Sigma_{0,n+1})$ is also denoted by $\s_i$. 
When $n=3$, 
$\s_1\in \MCG(\Sigma_{0, 4})$ fixes the punctures $a_\infty$ and $a_ 3$ and interchanges the punctures $a_1$ and $a_2$ counterclockwise. 
Similarly, $\sigma_2^{-1} \in \MCG(\Sigma_{0, 4})$ fixes $a_\infty$ and $a_1$ and interchanges $a_2$ and $a_3$ clockwise. 
See Figure~\ref{fig_action3braid}. 
\begin{figure}[htbp]
\begin{center}
\includegraphics[height=3cm]{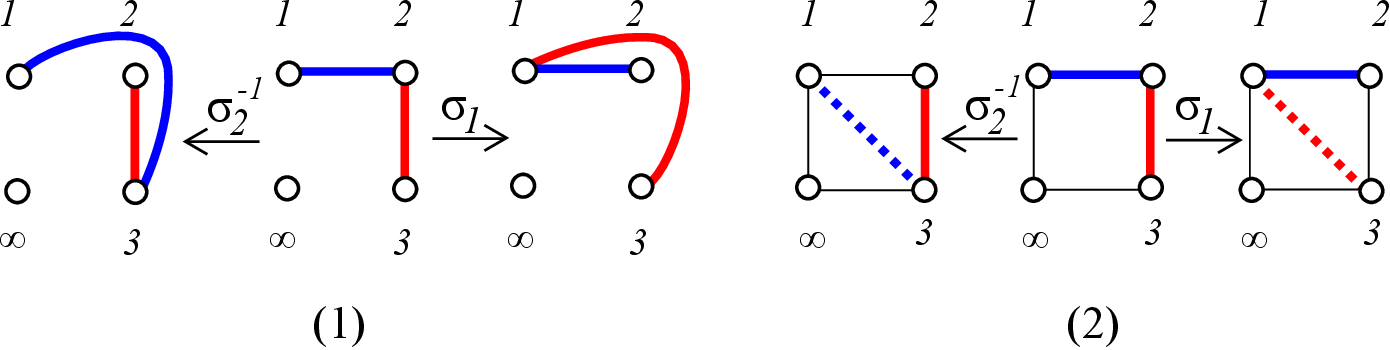}
\caption{
(1), (2) $\sigma_1, \sigma_2^{-1}: \Sigma_{0,4} \rightarrow \Sigma_{0,4}$. 
(Labels $1, 2, 3$ and $\infty$ represent punctures $a_1, a_2, a_3$ and $a_{\infty}$.) 
In (2) we regard $\Sigma_{0,4}$ as a square pillowcase with the corners removed.}
\label{fig_action3braid} 
\end{center}
\end{figure}

\subsection{Agol cycles of pseudo-Anosov $3$-braids}

To compute Agol cycles of pseudo-Anosov $3$-braids, 
we use the measured train track  $(\omega_0, \vxy)$  in  $\Sigma_{0,4}$ defined in Figure~\ref{fig_switch-omega}-(3). 
Notice that $(\omega_0, \left(\begin{smallmatrix} x \\ y \end{smallmatrix} \right))$ 
is triply weighted in the sense that the weights of the six branches are either $x, y$ or $x+y$. 
The measure is represented by the column vector 
$\left(\begin{smallmatrix} x \\ y \end{smallmatrix} \right)$. 
There are two large branches with the same weight $x+y$, 
and others are small branches.

For the proof of Theorem~\ref{theorem-braid-intro} 
it is enough to prove the following result. 

\begin{theorem}\label{theorem-braid}
Let $A = L^{q_k} R^{p_k} \cdots L^{q_1} R^{p_1}  \in \SL(2; \Z)$ 
where $p_1, q_1, \cdots, p_k, q_k$ and $k$ are positive integers.
Let $\ell= p_1 + q_1 + \cdots + p_k + q_k$ and $s$ be the slope of the eigenvectors with respect to the expanding eigenvalue $\lambda>1$ of $A$. 
For the pseudo-Anosov map
$$\phi_A = 
\Gamma(\s_1^{p_1} \s_2^{-q_1} \cdots \s_1^{p_k} \s_2^{-q_k})=
\s_2^{-q_k} \circ \s_1^{p_k} \circ \cdots \circ \s_2^{-q_1} \circ \s_1^{p_1}: \Sigma_{0,4} \rightarrow \Sigma_{0,4}$$ 
we have the following. 
\begin{enumerate}
\item
The measured train track 
$(\omega_0, \nu_0):= (\omega_0, \left(\begin{smallmatrix} 1\\ s \end{smallmatrix}\right))$ (see Figure~\ref{fig_switch-omega}-(3)) 
is suited to the stable measured lamination of $\phi_A$. 

\item
Starting with the measured train track 
$ (\omega_0, \nu_0)$, 
the first $\ell+1$ terms 
$$(\omega_0, \nu_0)
\ls^{q_k} \ 
\rs^{p_k} \ 
\cdots \ 
\ls^{q_1} \ 
\rs^{p_1}
(\omega_\ell, \nu_\ell)
$$
of the maximal splitting sequence satisfies
$(\omega_{\ell}, \nu_{\ell})= \phi_A(\omega_0, \lambda^{-1} \nu_0)$. 
Thus, it forms a length $\ell$ Agol cycle of $\phi_A$. 
\end{enumerate}
\end{theorem}

The rest of this section is devoted to prove Theorem~\ref{theorem-braid}. 
In the proof we will see that the above measured train track $(\omega_i, \nu_i)$ is equal to $\pi \circ \rho^{-1} (\tau_i, \mu_i)$, where 
$\pi: \Sigma_{1,4} \rightarrow \Sigma_{0,4}$ is a double covering, 
$\rho: \Sigma_{1,4} \rightarrow \Sigma_{1,1}$ is a $4$-fold covering, and 
$(\tau_i, \mu_i)$ is the measured train track in $\Sigma_{1,1}$ 
for $f_A: \Sigma_{1,1} \rightarrow \Sigma_{1,1}$ that appears in Theorem~\ref{theorem-2}-(2).
As a result, each $(\omega_i, \nu_i)$ has exactly two large branches whose weights are the same, and is triply weighted. 
Importance of triply weighted train tracks in $\Sigma_{0,4}$ in the study of Agol cycles was first pointed by Aceves and Kawamuro  \cite{AcevesKawamuro23}.

We start with taking four points 
$a_\infty:=(0, 0)$,  
$a_1:=(0, 1/2)$, 
$a_2:=(1/2, 1/2)$, and 
$a_3:=(1/2, 0)$ in $ \T= \R^2/\Z^2$. 
The homeomorphism $\f_A : \T \to  \T$ induced by $A \in \SL(2; \Z)$ 
fixes $a_{\infty}$ and permutes the three points $a_1, a_2$ and $a_3$. 
The involution $\f_ {-\I}: \T \to  \T$ induced by $-\I= \left(\begin{smallmatrix} -1 & 0 \\ 0 & -1 \end{smallmatrix}\right)$ 
fixes the four points $a_\infty, a_1, a_2, a_3$. 
For $A \in \SL(2; \Z)$ let 
$$g_A : \Sigma_{1, 4} \to \Sigma_{1, 4}; \vxy \mapsto A\vxy$$ 
be the restriction of $\f_A$ to the subspace $\Sigma_{1, 4}=\T \setminus \{a_\infty, a_1, a_2, a_3\}$, the $4$-punctured torus. 
In particular, $g_{-\I}:\Sigma_{1, 4} \to \Sigma_{1, 4}$ is an involution. 
The action $\Z_2 \ni 1 \mapsto g_{-\I} \in {\rm Homeo}_+(\Sigma_{1, 4})$ induces a double covering map 
$$\pi: \Sigma_{1, 4} \to \Sigma_{1, 4}/\Z_2 = \Sigma_{0, 4}.$$

\begin{proposition}\label{prop:pi}
Let $g_L, g_R: \Sigma_{1,4} \rightarrow \Sigma_{1,4}$ be the maps induced by $L, R \in  \SL(2; \Z)$, respectively. 
Then $\s_1,  \s_2^{-1}: \Sigma_{0,4} \rightarrow \Sigma_{0,4}$ 
satisfy 
\begin{eqnarray*}
\pi \circ g_L &=& \s_2^{-1} \circ \pi, \\
\pi \circ g_R &=& \s_1 \circ \pi.
\end{eqnarray*}
\end{proposition}


\begin{proof}
Recall that $g_L$, $g_R$ are restrictions of $\f_L$, $\f_R: \T \rightarrow \T$ 
to $\Sigma_ {1,4} \subset \Sigma_ {1,0}= \T$, respectively. 
The assertion follows from Figure~\ref{fig_g_map}. 
\end{proof}
\begin{figure}[htbp]
\begin{center}
\includegraphics[height=7cm]{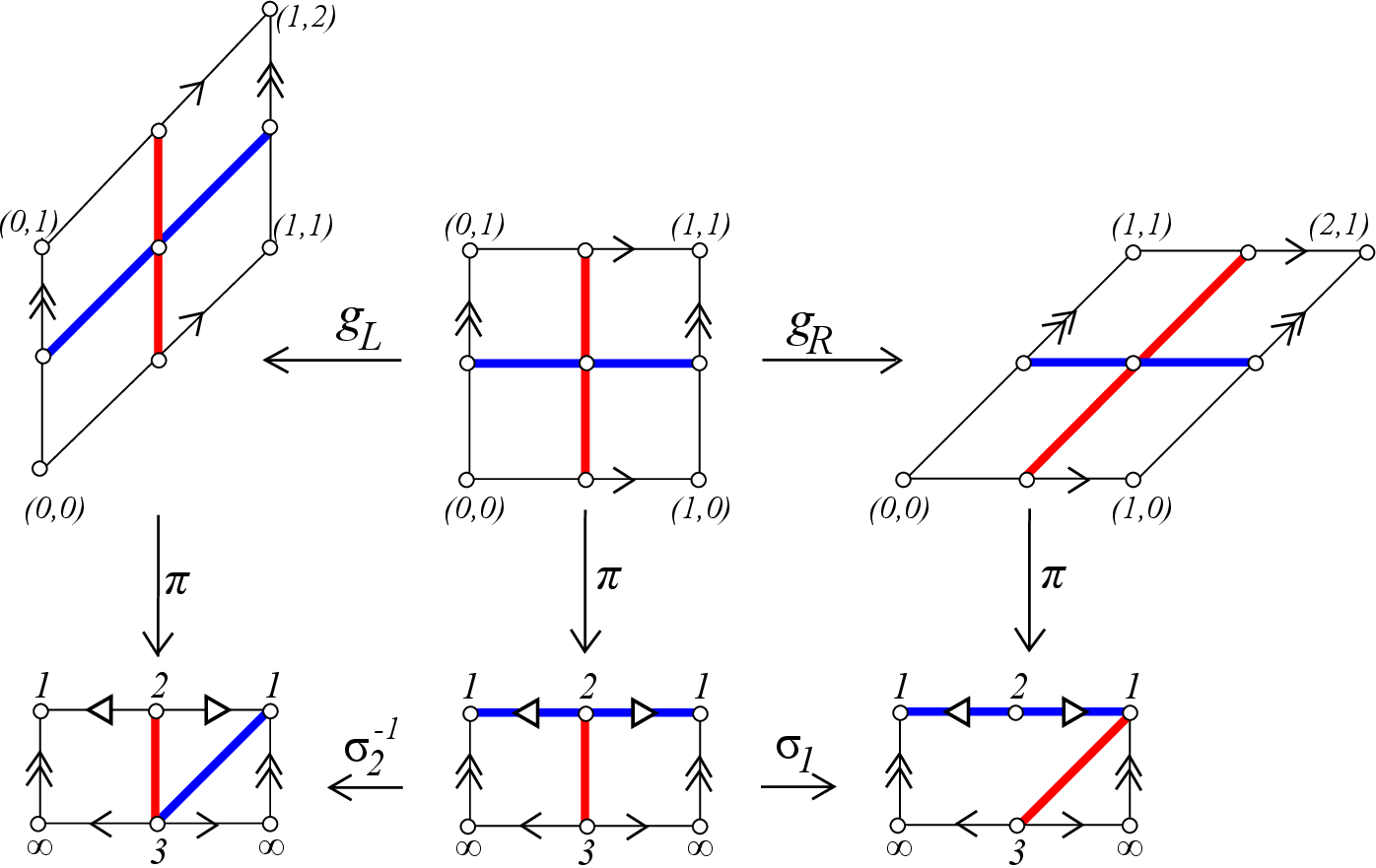}
\caption{Proof of Proposition~\ref{prop:pi}. 
See also Figure~\ref{fig_action3braid}.}
\label{fig_g_map} 
\end{center}
\end{figure}

\begin{lemma}
\label{lem:same_dilatation}
Let $A = L^{q_k} R^{p_k} \cdots L^{q_1} R^{p_1}  $ be as in Theorem~\ref{theorem-braid}. 
Consider 
$$\phi_A= \s_2^{-q_k} \circ \s_1^{p_k} \circ \cdots \circ \s_2^{-q_1} \circ \s_1^{p_1} : \Sigma_{0,4} \rightarrow \Sigma_{0,4}$$ 
induced by $A$. 
Then we have the following. 
\begin{enumerate}
\item 
$\pi \circ g_A = \phi_A \circ \pi$, 
where $\pi: \Sigma_{1,4} \rightarrow \Sigma_{0,4}$ is the double covering map. 

\item 
$g_A: \Sigma_{1,4} \rightarrow \Sigma_{1,4}$ and $\phi_A: \Sigma_{0,4} \rightarrow \Sigma_{0,4}$ 
are pseudo-Anosov maps with the same dilatation 
which is equal to the expanding eigenvalue $\lambda>1$ of $A$. 

\end{enumerate}
\end{lemma}

\begin{proof}
Statement (1) follows from Proposition \ref{prop:pi}. 
For Statement (2) 
we recall the Anosov map $\f_A: \T \rightarrow \T$ induced by $A$. 
The dilatation of $\f_A$ is equal to the expanding eigenvalue $\lambda$ of $A$. 
Since $g_A: \Sigma_{1,4} \rightarrow \Sigma_{1,4}$ is the restriction of $\f_A$ to 
$\Sigma_{1,4} \subset \Sigma_{1,0}= \T$, 
one sees that  $g_A$ is a pseudo-Anosov map and $\f_A$ and $g_A$ have the same dilatation.
Let $\mathcal{F}^{\tt s}$ and $\mathcal{F}^{\tt u}$  be the stable and unstable foliations with respect to $g_A$. 
Since $g_A: \Sigma_{1,4} \rightarrow \Sigma_{1,4}$ is a  pull-back 
$\phi_A: \Sigma_{0,4} \rightarrow \Sigma_{0,4}$ under $\pi: \Sigma_{1,4} \rightarrow \Sigma_{0,4}$,  
the images 
$\pi(\mathcal{F}^{\tt s})$ and $\pi(\mathcal{F}^{\tt u})$ under $\pi$ give the stable and unstable foliation with respect to $\phi_A$. 
Hence $\phi_A$ is also a pseudo-Anosov map with the same dilatation as that of $g_A$.
Thus the dilatation of $\phi_A$ is also the expanding eigenvalue $\lambda$ of $A$. 
This completes the proof. 
\end{proof}

Next we introduce 
the measured train track $(\mathcal T_0, \vxy)$ in $\Sigma_{1, 4}$ defined as the preimage of the measured train track $(\omega_0, \vxy)$  
under the double covering map $\pi: \Sigma_{1, 4}\to \Sigma_{0, 4}$ (Figure~\ref{fig_cover}); 
\begin{equation}
\label{eq:T_Omega}
(\mathcal T_0, \vxy):=\pi^{-1}(\omega_0, \vxy).
\end{equation}

We note that $(\mathcal T_0, \vxy)$ can be obtained by four copies of the measured train track $(\tau_0, \vxy)$ 
in  $\Sigma_{1,1}$ as in Figure~\ref{fig_cover}. 
To see this, consider the $4$-fold covering $\overline{\rho}: \T \to \T$ corresponding to the subgroup $2\Z\oplus2\Z$ of the fundamental group $\pi_1(\T)= \Z \oplus \Z$;   that is, the subgroup $\overline{\rho}_*(\pi_1(\T)) < \pi_1(\T)$ is isomorphic to $2\Z\oplus2\Z$ and its deck transformation group is $\Z/2\Z \oplus \Z/2\Z$. 
The covering map $\overline{\rho}$ induces a $4$-fold covering map 
$$\rho: \Sigma_{1, 4} \to \Sigma_{1,1}.$$ 
It satisfies 
$\rho\circ g_A=f_A\circ\rho \ \ \mbox{for any}\ \  A \in SL(2; \Z).$
Hence the preimage of $(\tau_0, \vxy)$ is 
\begin{equation}\label{eq:T=}
(\mathcal T_0, \vxy) = \rho^{-1}(\tau_0, \vxy)
\end{equation}
and $g_A: \Sigma_{1,4} \rightarrow \Sigma_{1,4}$ is a pull-back of 
$f_A: \Sigma_{1,1} \rightarrow \Sigma_{1,1}$ under $\rho$. 

\begin{figure}[htbp]
\centering
\includegraphics[height=8cm]{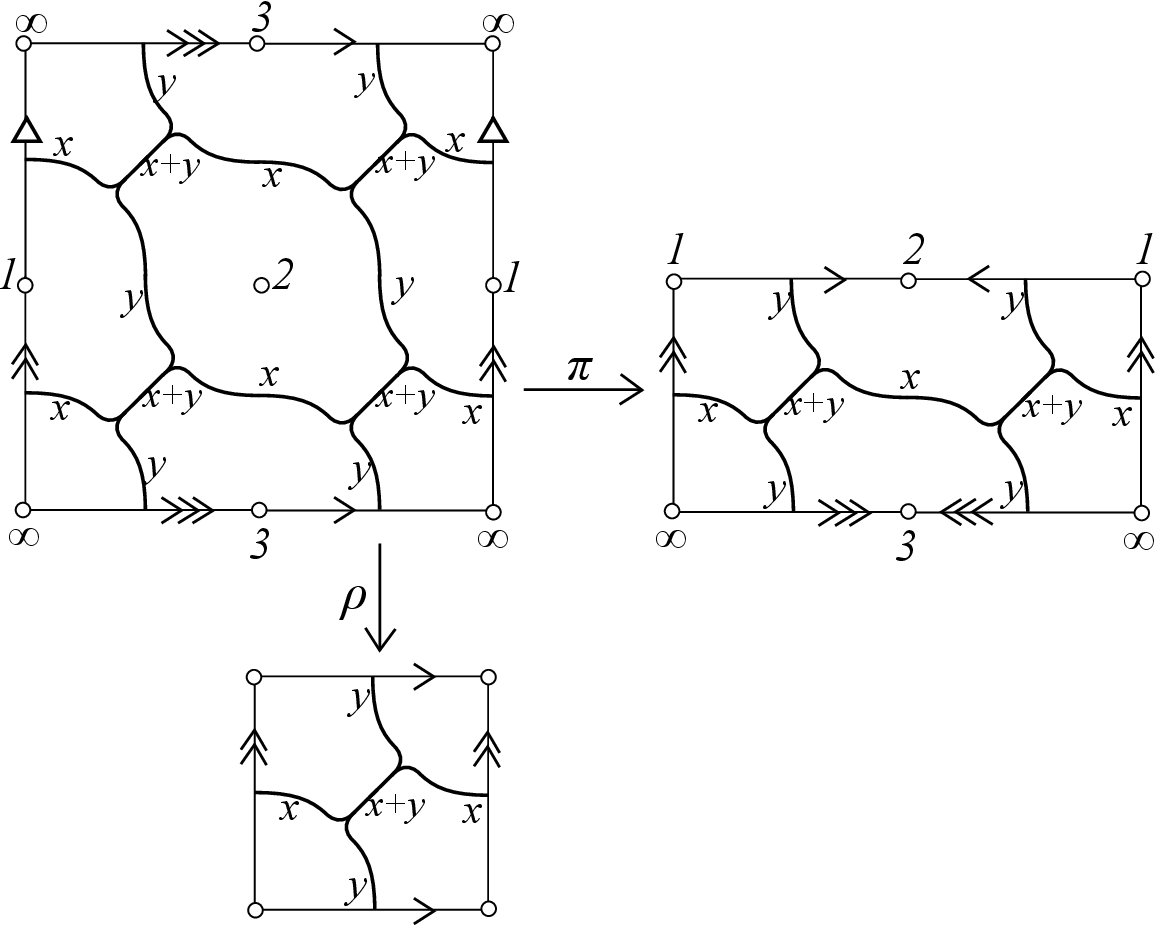}
\caption{The double covering map $\pi: \Sigma_{1, 4}\to \Sigma_{0, 4}$ and 
the $4$-fold covering map $\rho: \Sigma_{1, 4} \to \Sigma_{1,1}$.}
\label{fig_cover} 
\end{figure}

\begin{proposition}\label{pro:A}
For the measured train tracks $(\mathcal{T}_0, \vxy)$ in $\Sigma_{1,4}$ and 
$(\omega_0, \vxy)$ in $\Sigma_{0,4}$ 
we have the following. 
\begin{enumerate}
\item
If $y>qx$ then  both $(\mathcal{T}_0, \vxy)$ and 
$(\omega_0, \vxy)$ admit $q$ left splittings consecutively: 
$$(\mathcal{T}_0, \vxy) \ls^q  g_L^q (\mathcal{T}_0 , L^{-q}\vxy) 
\ \mbox{and}\ 
(\omega_0, \vxy) \ls^q \ \sigma_2^{-q}(\omega_0,  L^{-q}\vxy).$$
Moreover $\pi \circ \ls^q (\mathcal{T}_0, \vxy) = \ \ls^q \circ \pi (\mathcal{T}_0, \vxy) = \sigma_2^{-q}(\omega_0,  L^{-q}\vxy)$ as in the left commutative diagram below. 

\item
If $x>py$ then  both $(\mathcal{T}_0, \vxy)$ and 
$(\omega_0, \vxy)$ admits $p$ right splittings consecutively: 
$$(\mathcal{T}_0, \vxy) \rs^p g_R^p(\mathcal{T}_0,  R^{-p}\vxy) 
\ \mbox{and}\ 
(\omega_0, \vxy) \rs^p \ \sigma_1^p (\omega_0, R^{-p}\vxy).$$
Moreover $\pi \circ \rs^p (\mathcal{T}_0, \vxy) =\  \rs^p \circ \pi  (\mathcal{T}_0, \vxy) = \sigma_1^p (\omega_0, R^{-p}\vxy)$ as in the right commutative diagram below. 
\end{enumerate}
$$
\begin{array}{ccc}
(\mathcal T_0, \vxy) & \ls^q & g_L^q (\mathcal T_0 , L^{-q}\vxy)  \\ 
&&\\
\pi \big\downarrow && \pi \big\downarrow \\
&&\\
(\omega_0, \vxy) & \ls^q & \sigma_2^{-q}(\omega_0,  L^{-q}\vxy) 
\end{array} \hspace{1cm}
\begin{array}{ccc}
(\mathcal T_0, \vxy) & \rs^p & g_R^p (\mathcal T_0 , R^{-p}\vxy)  \\ 
&&\\
\pi \big\downarrow && \pi \big\downarrow \\
&&\\
(\omega_0, \vxy) & \rs^p & \sigma_2^{p}(\omega_0,  R^{-p}\vxy) 
\end{array}
$$
\end{proposition}

\begin{proof}
Suppose that $y > qx$. 
By Corollary~\ref{cor:tau-consequtive}-(1) 
we have $(\tau_0,\vxy ) \ls^q f_L^q(\tau_0, L^{-q}\vxy) $. 
Using the fact that 
$g_A: \Sigma_{1,4} \rightarrow \Sigma_{1,4}$ is a restriction of $f_A: \Sigma_{1,1} \rightarrow \Sigma_{1,1}$ 
and the description of  $(\mathcal T_0, \vxy) $ in  (\ref{eq:T=}), 
one can verify that 
$(\mathcal T_0, \vxy) \ls^q g_L^q (\mathcal T_0 , L^{-q}\vxy)$.

We turn to the proof of 
$(\omega_0, \vxy) \ls^q \ \sigma_2^{-q}(\omega_0,  L^{-q}\vxy)$. 
We forget the measure for a moment. 
Figure~\ref{fig_3braid_lrmap} explains 
$\omega_0  \ls \sigma_2^{-1}(\omega_0)$ and $\omega_0  \rs \sigma_1(\omega_0)$. 
As in the proof of Proposition~\ref{propA}-(1) 
one can prove that 
$\omega_0 \ls^q  \sigma_2^{-q}(\omega_0)$ and $\omega_0 \rs^p  \sigma_1^p (\omega_0)$. 
Now we consider the measured train track $(\omega_0, \vxy)$ under the assumption $y > qx$. 
It is not hard to see that 
$(\omega_0, \vxy) \ls \ \sigma_2^{-1}(\omega_0,  L^{-1}\vxy)$. 
Using  Lemma~\ref{lemma:0A} repeatedly 
one can prove that $(\omega_0, \vxy)$ admits $q$ left splittings consecutively and 
$(\omega_0, \vxy) \ls^q \ \sigma_2^{-q}(\omega_0,  L^{-q}\vxy)$. 
Then by Proposition~\ref{prop:pi}
\begin{eqnarray*}
\pi \circ \ls^q (\mathcal T_0, \vxy) &=& \pi \circ g_L^q (\mathcal T_0 , L^{-q}\vxy) =  \sigma_2^{-q}(\omega_0,  L^{-q}\vxy), \ 
\mbox{and}
\\
\ls^q \circ \pi (\mathcal T_0, \vxy)&=& \ \ls^q (\omega_0, \vxy)= \sigma_2^{-q}(\omega_0,  L^{-q}\vxy). 
\end{eqnarray*}
The proof of (1) is done.  
By a similar argument we can prove Statement (2). 
\end{proof}

\begin{figure}[htbp]
\begin{center}
\includegraphics[height=6.5cm]{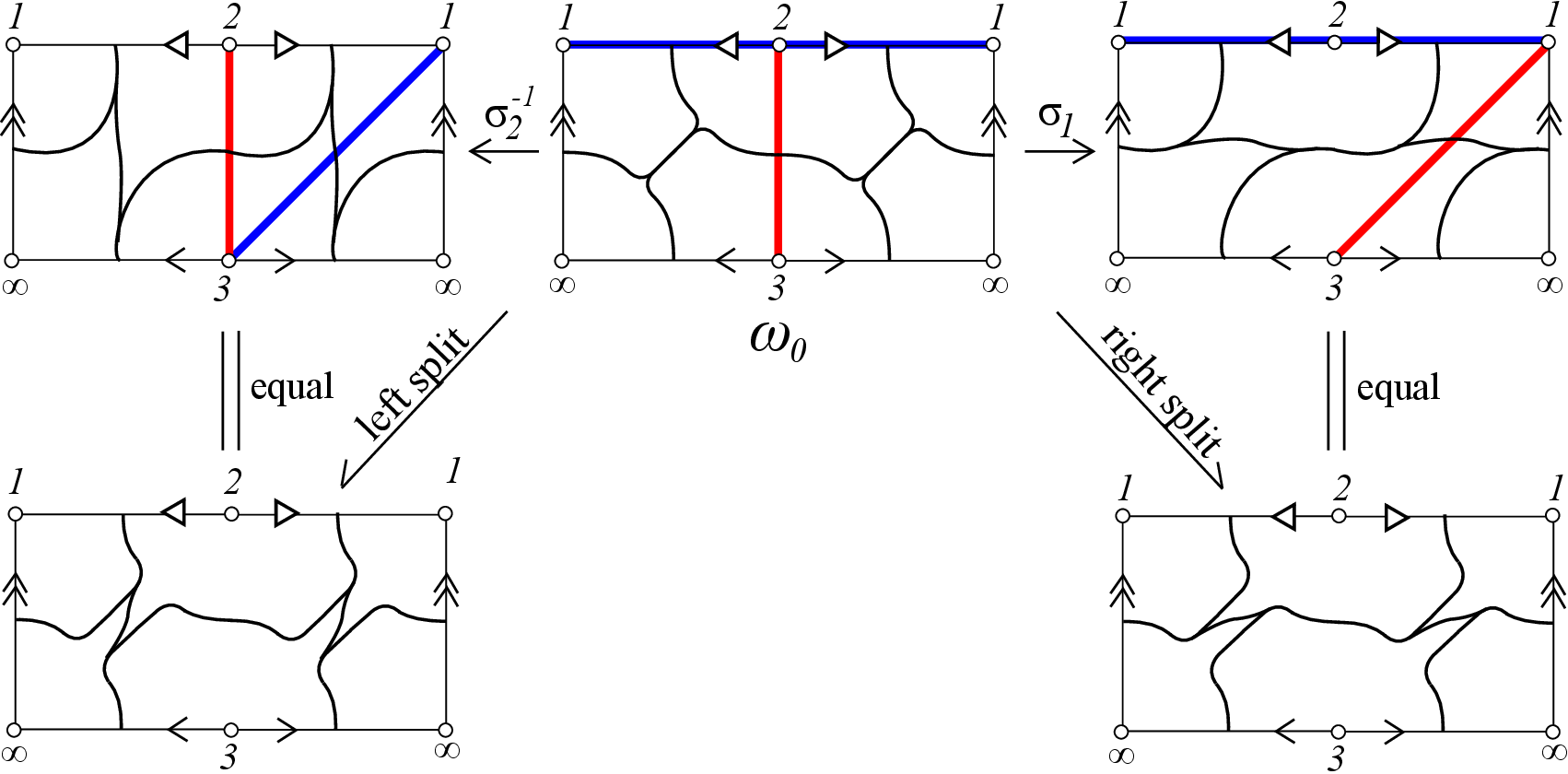}
\caption{Proof of Proposition~\ref{pro:A}: 
$\omega_0  \ls \sigma_2^{-1}(\omega_0)$ and $\omega_0  \rs \sigma_1(\omega_0)$.}
\label{fig_3braid_lrmap} 
\end{center}
\end{figure}

We now prove Theorem~\ref{theorem-braid}. 

\begin{proof}[Proof of Theorem~\ref{theorem-braid}]

By Theorem~\ref{theorem-2}-(1) 
$(\tau_0, \left(\begin{smallmatrix} 1 \\ s \end{smallmatrix}\right))$ is suited to the stable measured lamination of $f_A$. 
The pseudo-Anosov map $g_A: \Sigma_{1,4} \rightarrow \Sigma_{1,4}$ 
is a pull-back of $f_A: \Sigma_{1,1} \rightarrow \Sigma_{1,1}$ under the covering map $\rho: \Sigma_{1,4} \rightarrow \Sigma_{1,1}$, 
and $g_A$ is also  a pull-back of $\phi_A: \Sigma_{0,4} \rightarrow \Sigma_{0,4}$ 
under the covering map $\pi: \Sigma_{1,4} \rightarrow \Sigma_{0,4}$. 
Thus  the measured train track 
$(\omega_0, \left(\begin{smallmatrix} 1 \\ s \end{smallmatrix}\right)) = \pi \circ \rho^{-1} (\tau_0, \left(\begin{smallmatrix} 1 \\ s \end{smallmatrix}\right))$ 
is suited to the stable measured lamination of $\phi_A$. 

Theorem~\ref{theorem-2} states that 
$(\tau_0, \mu_0 = \left(\begin{smallmatrix} 1 \\ s \end{smallmatrix}\right))
\ls^{q_k} \ 
\rs^{p_k} \ 
\cdots \ 
\ls^{q_1} \ 
\rs^{p_1}
(\tau_\ell, \mu_\ell)
$
 forms a length $\ell$ Agol cycle of $f_A$. 
 Note that similar commutative diagrams as in Proposition~\ref{pro:A} hold for the pair of measured train tracks $(\mathcal{T}_0, \left(\begin{smallmatrix} x \\ y \end{smallmatrix}\right))$ and 
 $(\tau_0, \left(\begin{smallmatrix} x \\ y \end{smallmatrix}\right))$. 
 Moreover by Lemma~\ref{lem:same_dilatation}-(2) the dilatations of $f_A$, $g_A$ and $\phi_A$ are the same which is the expanding eigenvalue of $A$.  
 These facts together with the property $\rho \circ g_{A'} = f_{A'} \circ \rho$ for each $A' \in \SL(2; Z)$ imply that 
$$(\mathcal{T}_0, \left(\begin{smallmatrix} 1 \\ s \end{smallmatrix}\right)) = \rho^{-1}(\tau_0, \mu_0) 
\ls^{q_k} \ 
\rs^{p_k} \ 
\cdots \ 
\ls^{q_1} \ 
\rs^{p_1}
\rho^{-1}(\tau_\ell, \mu_\ell)$$
gives a length $\ell$ Agol cycle of $g_A$. 
Each measured train track of the above Agol cycle of $g_A$ is of the form $ \rho^{-1}(\tau_i, \mu_i)$. 
By Proposition~\ref{propA1}-(2)  $\tau_i$ is preserved by the involution 
$f_{-\I}: \Sigma_{1,1} \rightarrow \Sigma_{1,1}$. 
Thus $\rho^{-1}(\tau_i) $ is also preserved by the involution $g_{-\I}: \Sigma_{1,4} \rightarrow \Sigma_{1,4}$. 
This means that $\pi\circ \rho^{-1}(\tau_i)$ descends to a train track in the quotient space $\Sigma_{0,4}$.

 On the other hand, 
 we have the commutative diagrams in Proposition~\ref{pro:A}.  
 The above Agol cycle of $g_A$ together with Proposition~\ref{prop:pi} tells us that 
 $$(\omega_0, \nu_0) 
 \ls^{q_k} \ 
\rs^{p_k} \ 
\cdots \ 
\ls^{q_1} \ 
\rs^{p_1}
(\omega_{\ell}, \nu_{\ell}),$$ 
where 
$(\omega_i, \nu_i) = \pi \circ \rho^{-1} (\tau_i, \mu_i)$ where $i= 1, \dots, \ell$ 
gives a length $\ell$ Agol cycle of $\phi_A$. 
This completes the proof. 
\end{proof}

\begin{example}[The pseudo-Anosov $3$-braid $\sigma_1 \sigma_2^{-1}$]
\label{ex_lr_braid}
Let $A = LR $ and $s =  \tfrac{1+ \sqrt{5}}{2} $ be as in Example~\ref{ex_lr_torus}. 
Then 
$(\omega_0, \nu_0= \left(\begin{smallmatrix} 1 \\ s \end{smallmatrix}\right)) \ls (\omega_1, \nu_1) \rs (\omega_2, \nu_2)$  
forms a length $2$ Agol cycle of  $\phi_A= \sigma_2^{-1} \circ \sigma_1: \Sigma_{0,4} \rightarrow \Sigma_{0,4}$. 
See Figure~\ref{fig_lr_braid}. 
\end{example}

\begin{figure}[htbp]
\centering
\includegraphics[height=3.2cm]{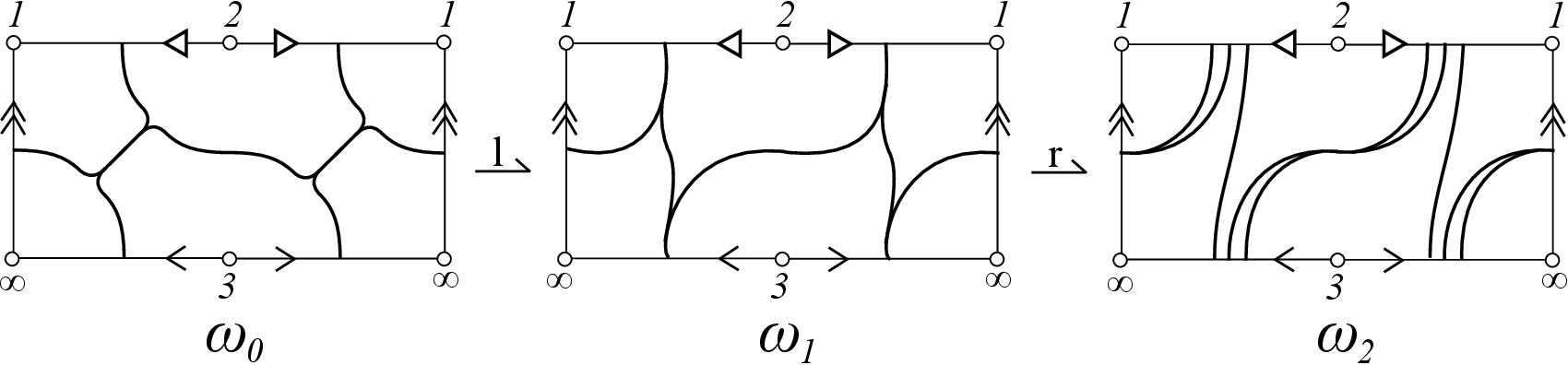}
\caption{
Example~\ref{ex_lr_braid}: $\omega_0 \ls \omega_1 \rs \omega_2$. 
See also Figure~\ref{fig_lr_torus}.}
\label{fig_lr_braid} 
\end{figure}

\begin{corollary}\label{cor:A}
Pseudo-Anosov $3$-braids $\beta$ and $\beta'$ are conjugate in $B_3/Z(B_3)$ where $Z(B_3)$ is the center of $B_3$ 
if and only if their Agol cycles are equivalent. 
\end{corollary}

\begin{proof}
The only-if-part follows from Theorem~\ref{theorem_Hodgson-Issa-Segerman}. For the if-part we prove the contrapositive statement. 
Assume that pseudo-Anosov 3-braids $\beta$ and $\beta'$ are not conjugate in $B_3/Z(B_3)$. 
Due to Murasugi's classification we may assume that $\beta$ is conjugate to $\Delta^{2j} \s_1^{p_1} \s_2^{-q_1} \cdots \s_1^{p_k} \s_2^{-q_k}$ and $\beta'$ is conjugate to $\Delta^{2j'} \s_1^{p'_1} \s_2^{-q'_1} \cdots \s_1^{p'_l} \s_2^{-q'_l}$ for some $j, k, p_1, q_1, \dots, p_k, q_k$ and $j', l, p'_1, q'_1, \dots$, $p'_l$, $q'_l$.
If their Agol cycles were equivalent then by Theorem~\ref{theorem-braid-intro} and Lemma~\ref{lemma:0A} the cyclically-ordered sets $\{(p_1, q_1), \cdots, (p_k, q_k)\}$ and $\{(p'_1, q'_1), \cdots, (p'_l, q'_l)\}$ had to be equal, which means that  $\beta$ and $\beta'$ are conjugate in $B_3/Z(B_3)$. This is a contradiction. 
\end{proof}

\subsection{Garside canonical lengths v.s. Agol cycle lengths}
\label{subsection:Garside_canonical_length}

The {\em Garside (left) normal form} of an $n$-braid is used to improve Garside's solution to the conjugacy problem for $B_n$ \cite{El-RifaiMorton94,Epstein92}. 
Garside introduced the fundamental braid $\Delta$ 
which is a positive half-twist on all of the $n$ strands. 
The {\em simple} elements in $B_n$ are the positive braids in which every pair of strands cross at most once. 
There are $n!$ simple elements in $B_n$.  
Given a braid $b \in B_n$, its Garside normal form is a special representation by a braid word of the form $\Delta^r P_1 P_2 \cdots P_s$ for some integer $r$ and simple elements $P_1, \cdots, P_s$ satisfying 
the condition that each $P_i P_{i+1}$ is {\em left-weighted} 
 introduced by Elrifai and Morton \cite{El-RifaiMorton94}. 
The number of simple elements in the Garside normal form is called the  {\em canonical length} of $b$. 
The super summit set ${\rm SSS}(b)$ is a finite set of $n$-braids with minimal canonical length in the conjugacy class of $b$. 
The conjugacy problem in $B_n$ can be solved by computing the super summit set ${\rm SSS}(b)$. 

For 3-braids, 
the fundamental element is $\Delta = \sigma_1 \sigma_2 \sigma_1 = \sigma_2 \sigma_1 \sigma_2 $. 
The Garside normal form of the pseudo-Anosov $3$-braid $\s_1^p \s_2^{-q}$ for positive integers $p$ and $q$ is given as follows. 
\begin{eqnarray*}
&\Delta^{-q}& \hspace{-3mm} \sigma_2^p \cdot (\sigma_2 \sigma_1) \cdot (\sigma_1 \sigma_2)  \cdots   (\sigma_2 \sigma_1) 
\hspace{17mm} \mbox{if\ } q \ \mbox{is odd}, 
\\
&\Delta^{-q}& \hspace{-3mm} \sigma_1^p \cdot (\sigma_1 \sigma_2) \cdot (\sigma_2 \sigma_1)      \cdots  (\sigma_1 \sigma_2) \cdot (\sigma_2 \sigma_1)  
\hspace{3mm} \mbox{if\ } q \ \mbox{is even}, 
\end{eqnarray*}
where the simple elements $\s_2\s_1$ and $\s_1\s_2$ alternate in the tail and 
the number of the simple elements $\s_2\s_1$ or $\s_1\s_2$ is $q$. 
This can be obtained by replacing $\sigma_2^{-1}$ with $\Delta^{-1} \sigma_2 \sigma_1$, and shift $\Delta^{-1}$ to the left. 
For example when $q=1$ and $2$,  
we have 
$\sigma_1^p \sigma_2^{-1}= \Delta^{-1} \sigma_2^p \cdot (\sigma_2 \sigma_1)$ 
and 
$\sigma_1^p \sigma_2^{-2} = \Delta^{-2} \sigma_1^p \cdot (\sigma_1 \sigma_2) \cdot (\sigma_2 \sigma_1)$. 
Thus, the canonical length of  $\s_1^p \s_2^{-q}$  is $p+q$.

Similarly, a direct computation  shows that 
the Garside canonical length of the  pseudo-Anosov $3$-braid 
$\beta= \Delta^{2m} \s_1^{p_1} \s_2^{-q_1} \cdots \s_1^{p_k} \s_2^{-q_k}$ is $ p_1+q_1+\cdots+p_k+q_k$ which is by Theorem~\ref{theorem-braid} exactly the Agol cycle length of $\beta$. 
It is known that $\beta$ belongs to its super summit set,  see 
 Aguilera \cite{Aguilera15} for example. 
Therefore, we obtain the following result.

\begin{theorem}\label{thm:Garside}
For every pseudo-Anosov $3$-braid $\beta$, the Agol cycle length of $\beta$, the Garside canonical length of any element in the super summit set ${\rm SSS}(\beta)$ are the same.
 \end{theorem}

The same result as Theorem~\ref{thm:Garside} does not hold for higher braid index. 
For example consider the $5$-braid $\alpha:=  \sigma_3 \sigma_2 \sigma_3 \sigma_4  \sigma_1^{-1}$. 
By \cite[Section 5]{Agol11} the Agol cycle length of $\alpha$ is $6$ but the Garside canonical length of any element in ${\rm SSS}(\alpha)$ is $2$. 
(One can use a computer program {\em Braiding} by Gonz\'{a}lez-Meneses  to compute the super summit set \cite{Gonzalez-Meneses04}.) 
It would be interesting to explore the relation between Agol cycle lengths and Garside canonical lengths for general pseudo-Anosov braids.

For the dual-version of the Garside normal form, also known as the Birman-Ko-Lee normal form, we confirm the same result as Theorem~\ref{thm:Garside}. 
For the above $5$-braid $\alpha$, it's dual Garside canonical length of any super summit element is $3$.

\section*{Acknowledgement}
The authors thank Ian Agol, Fran\c{c}ois Gu\'{e}ritaud, Ahmad Issa, Dan Margalit, Makoto Sakuma,  Saul Schleimer, and Henry Segerman for valuable comments. 
Kawamuro's research was supported by NSF DMS-2005450.
Kin's research was  supported by JSPS KAKENHI Grant Numbers JP21K03247, JP22H01125, JP23H01081.

\bibliographystyle{hamsplain}
\bibliography{biblio}

\end{document}